\documentclass[sn-mathphys,Numbered]{sn-jnl}% Math and Physical Sciences Reference Style
\usepackage{graphicx}%
\usepackage{multirow}%
\usepackage{amsmath,amssymb,amsfonts}%
\usepackage{amsthm}%
\usepackage{mathrsfs}%
\usepackage[title]{appendix}%
\usepackage{xcolor}%
\usepackage{textcomp}%
\usepackage{manyfoot}%
\usepackage{booktabs}%
\usepackage{algorithm}%
\usepackage{algorithmicx}%
\usepackage{algpseudocode}%
\usepackage{listings}%

\usepackage{enumitem}
\theoremstyle{thmstyleone}%
\newtheorem{theorem}{Theorem}[section]%  meant for continuous numbers

\newtheorem{proposition}[theorem]{Proposition}% 

\theoremstyle{thmstyletwo}%
\newtheorem{example}[theorem]{Example}%
\newtheorem{remark}[theorem]{Remark}%

\theoremstyle{thmstylethree}%
\newtheorem{definition}[theorem]{Definition}%

\newtheorem{mtheorem}{Theorem}

\newtheorem{defrem}[theorem]{Definition and Remark}
\newtheorem{corollary}[theorem]{Corollary}
\newtheorem{lemma}[theorem]{Lemma}

\newcommand{\R}{\mathbb{R}}
\newcommand{\Z}{\mathbb{Z}}
\newcommand{\N}{\mathbb{N}}

\newcommand{\Q}{\mathbb{Q}}

\newcommand{\qf}[1]{\langle #1\rangle}

\newcommand{\Gcal}{\mathcal{G}}
\newcommand{\Gcalred}{\mathcal{G}^{\text{red}}}
\newcommand{\red}[1]{#1^{\text{red}}}

\DeclareMathOperator{\cha}{char }

\newcommand{\aut}{\operatorname{aut}}
\newcommand{\sym}{\operatorname{sym}}
\newcommand{\og}{\operatorname{O}}
\newcommand{\iso}{\operatorname{iso}}

\newcommand{\spn}{\operatorname{span}}

\begin{document}

\title[Cliques in Representation Graphs of Quadratic Forms]{Cliques in Representation Graphs of Quadratic Forms}

%%=============================================================%%
%% Prefix	-> \pfx{Dr}
%% GivenName	-> \fnm{Joergen W.}
%% Particle	-> \spfx{van der} -> surname prefix
%% FamilyName	-> \sur{Ploeg}
%% Suffix	-> \sfx{IV}
%% NatureName	-> \tanm{Poet Laureate} -> Title after name
%% Degrees	-> \dgr{MSc, PhD}
%% \author*[1,2]{\pfx{Dr} \fnm{Joergen W.} \spfx{van der} \sur{Ploeg} \sfx{IV} \tanm{Poet Laureate} 
%%                 \dgr{MSc, PhD}}\email{iauthor@gmail.com}
%%=============================================================%%

\author*[1,2]{\fnm{Nico} \sur{Lorenz}}\email{nico.lorenz@ruhr-uni-bochum.de}

\author[2,3]{\fnm{Marc Christian} \sur{Zimmermann}}\email{marc.christian.zimmermann@gmail.com}
\equalcont{These authors contributed equally to this work.}

\affil*[1]{\orgdiv{Fakult\"at f\"ur Mathematik}, \orgname{Ruhr-Universit\"at Bochum}, \orgaddress{\street{Universit\"atsstra{\ss}e 150}, \city{Bochum}, \postcode{44780}, \state{North Rhine-Westphalia}, \country{Germany}}}

\affil[2]{\orgdiv{Department Mathematik/Informatik, Abteilung Mathematik}, \orgname{Universit\"at zu K\"oln}, \orgaddress{\street{Weyertal~86--90}, \city{K\"oln}, \postcode{50931}, \state{North Rhine-Westphalia}, \country{Germany}}}

%%==================================%%
%% sample for unstructured abstract %%
%%==================================%%

\abstract{We study cliques in graphs arising from quadratic forms where the vertices are the elements of the module of the quadratic form and two vertices are adjacent if their difference represents some fixed scalar. 
    We determine structural properties and the clique number for quadratic forms over finite rings. We further extend previous results about graphs arising from such forms and forms over fields of characteristic 0 in a unified framework.}



\maketitle

\section{Introduction}

\subsection{History}

The main object of interest in this paper are graphs arising from quadratic forms by taking the elements of the underlying module of the quadratic form $q$ as vertices and including an edge between module elements $x$ and $y$ if and only if the difference $x-y$ represents some fixed scalar $a$, i.e. we have $q(x-y) = a$.

Graphs of this kinds have appeared as distance graphs of quadratic forms over finite fields to generalize distance graphs in standard euclidean space, on the sphere, or in hyperbolic space.

There the spectrum of such graphs has been studied. To our knowledge this started with the sums-of-squares form over finite fields, where the spectrum has been computed in \cite{medrano_myers_stark_terras_1996}. 
Now the sum-of-squares form is a particular choice of a regular quadratic form over a finite field, and as such not unique: in every fixed dimension there are precisely two such forms if the characteristic of the field is not $2$ and there are two such in every even dimension if the characteristic is $2$. 
Extensions of this work include a generalization to all such forms \cite{bannai_shimabukuro_tanaka_2009} and some finite rings instead of fields \cite{medrano:1998:feg}.

These results can, for example, be used to derive lower bounds on the chromatic number of the associated graphs using the Lovasz theta number, or the Hoffman bound, in complete analogy to the case of geometric distance graphs.

We are interested in the related question of determining the clique number of a graph. While this also gives a lower bound on the chromatic number, as all elements of a clique have to be colored differently, we find this problem to be interesting in its own right. Recently Krebs (cf. \cite{Krebs_2022} computed the clique number for the sum-of-squares form over finite fields, building on results of Chilakamarri (cf. \cite{Chilakamarri_1988}) for the sum-of-squares form over the rationals.

\subsection{Statement of the problem}

We now give a precise statement of the problem.
Let $A$ be a ring and $V$ be an $A$-module. 
A \emph{quadratic form} on $V$ is a map $q:V\to A$ such that
    \begin{enumerate}
        \item[(QF1)] For all $a\in A$ and $x\in V$ we have $q(ax)=a^2q(x)$
        \item[(QF2)] The map $b_q:V\times V\to A, (x,y)\mapsto q(x+y)-q(x)-q(y)$ is bilinear.
    \end{enumerate}
In particular we have $b_q(x,x)=2q(x)$.
We will only consider \emph{non-degenerate} quadratic forms, i.e. we have 
$$\{x\in V\mid b(x,y)=0\text{ for all } y\in V\}=\{0\}.$$
Given a quadratic form, we can define the following graph which is be the main object of interest in this paper.
\begin{definition}
	Let $q$ be a quadratic form on an $A$-module $V$ and $a\in A$. The \emph{representation graph} $\mathcal G_{q, a}=(V, E)$ is defined by its set of edges
	$$E:=\left\{\{x,y\} \in \binom{V}{2}\mid q(x-y) = a\right\}.$$
\end{definition}
As noted above these graphs have been considered by numerous authors, sometimes under the name of finite euclidean graphs or, if $a =1$ and $q$ is the sum-of-squares form, as unit-quadrance graphs. 

We are mostly interested in the size of a largest \emph{clique} in $\Gcal_{q,a}$, that is a set of vertices where any pair is adjacent which is of largest cardinality among all such sets. We denote the cardinality of a maximum cliqe by $\omega(\Gcal_{q,a})$. Concretely we are asking for the maximum number of elements in $V$ such that each pairwise difference represents $a$ under $q$. In addition we are interested in the number of such maximum cliques, which we denote by $\omega_{\max}(\Gcal_{q,a})$.

\subsection{Contribution and results}

For the sake of readability of this paper and to make our explicit results easily accessible we collect the statements of the main results at this time. 
Notation that is necessary to understand the statement is referred to.

Firstly, we computed the clique number of representation graphs associated to units in various cases generalizing Theorems 1.2 and 1.3 in \cite{Krebs_2022}.

For local rings of odd characteristic the value of $\omega(\Gcal_{q,a})$ depends on whether $q$ is isometric to a specific quadratic form $\gamma_{a,n}$, which we refer to as ``test form'', (for a definition see Definition and Remark~\ref{MaxCliqueForm}) and the characteristic of the residue field of $A$. 

For finite fields of characteristic $2$ it depends, again, on isometry with the test form $\gamma_{a,n}$ and, now, on $n$ being even or odd.

\begin{mtheorem}[Clique numbers]\label{Intro_thm:cliquenumbers}\phantom{a}

    \begin{enumerate}[label=(\alph*)]
        \item Let $A$ be a finite local ring whose residue field $K$ has odd characteristic and $q$ be an $n$-dimensional quadratic form over $A$ and $a\in A^\ast$. 

        Then $\omega(\Gcal_{q,a})$ is as depicted in Table~\ref{tbl:cliquenumbers:char:odd}.

        \item Let $A$ be a finite field of characteristic 2 and $q$ be an $n$-dimensional quadratic form over $A$ and $a\in A^\ast$. 

        Then $\omega(\Gcal_{q,a})$ is as depicted in Table~\ref{tbl:cliquenumbers:char:2}.
    \end{enumerate}
\end{mtheorem}

The proof of the theorem is given by Theorems~\ref{CliqueNumbersTheorem}, \ref{CliqueNumbersTheoremChar2}.

\begin{table}[!ht]
        \begin{tabular}{c|c|c|c}
            Case & isometry condition & characteristic condition &  $\omega(\Gcal_{q,a})$\\ \hline
            A & $q\not\cong \gamma_{a,n}$ & $\cha K\mid n$ & $n$ \\ \hline
            B & $q\not\cong \gamma_{a,n}$ & $\cha K\nmid n, n+1$ & $n$ \\ \hline
            C & $q\not\cong \gamma_{a,n}$ & $\cha K\mid n+1$ & $n+1$ \\ \hline
            D & $q\cong \gamma_{a,n}$ & $\cha K\nmid n+2$ & $n+1$ \\ \hline
            E & $q\cong \gamma_{a,n}$ & $\cha K\mid n+2$ & $n+2$
        \end{tabular}
    \caption{Clique numbers for finite local rings of odd characteristic.}
    \label{tbl:cliquenumbers:char:odd}
\end{table}

\begin{table}[!ht]
        \begin{tabular}{c|c|c|c}
            Case & isometry condition & dimension condition &  $\omega(\Gcal_{q,a})$\\ \hline
            A & $q\not\cong \gamma_{a,n}$ & $n\equiv 2\mod 4$ & $n-1$ \\ \hline
            B & $q\not\cong \gamma_{a,n}$ & $n\equiv 0\mod 4$ & $n$ \\ \hline
            C & $q\cong \gamma_{a,n}$ & $n\equiv 0\mod 4$ & $n+1$ \\ \hline
            D & $q\cong \gamma_{a,n}$ & $n\equiv 2\mod 4$ & $n+2$
        \end{tabular}
    \caption{Clique numbers for finite fields of characteristic $2$.}
    \label{tbl:cliquenumbers:char:2}
\end{table}

\newpage

Secondly, we computed the number of maximum cliques in various cases, this is done by an orbit-stabilizer computation with respect to the group of ``affine isometries'' $\iso(q) = F^n \rtimes \og(q)$ (see \eqref{eq:affine:isometries}).

\begin{mtheorem}[Number of maximum Cliques]\label{Intro_NumberOfCliques}
    Let $F$ be a finite field with $f$ elements, $q$ be an $n$-dimensional quadratic form over $F$ and $a\in F^*$. 
    We use the notation of Corollary~\ref{cor:SizeIso} and Theorem~\ref{NumberOfCliques}.
    Let $\alpha$ be defined by
    \begin{align*}
        \alpha=\begin{cases}
            2f-2,&\text{ if }q\perp \gamma_{a,n-2}\text{ is hyperbolic},\\
            2f+2,&\text{ otherwise.}
        \end{cases}
    \end{align*}
    The number of maximum cliques $\omega_\text{max}(\Gcal_{q,a})$, is given in
    \begin{enumerate}[label=(a)]
        \item Table~\ref{tbl:numbersofcliques:char:odd} if $\cha F\neq 2$.
    \item Table~\ref{tbl:numbersofcliques:char:2} if $\cha F=2$ .
    \end{enumerate}
\end{mtheorem}

The proof of the theorem is given by Theorem~\ref{NumberOfCliques}.

{
\renewcommand{\arraystretch}{1.5}
\begin{table}[!ht]
        \begin{tabular}{c||c|c|c}
            case / isometry & $q$ hyperbolic & $dim(q)$ even, $q$ not hyperbolic & $\dim(q)$ odd\\ \hline\hline 
            A & 
            $\frac{|\iso(q)|}{\alpha\cdot n!}$ & 
            $\frac{|\iso(q)|}{\alpha\cdot n!}$ & 
            $\frac{|\iso(q)|}{\alpha \cdot n!}$\\ \hline
            B & 
            $\frac{|\iso(q)|}{2\cdot n!}$ & 
            $\frac{|\iso(q)|}{2 \cdot n!}$ & 
            $\frac{|\iso(q)|}{2 \cdot n!}$\\ \hline
            C & 
            $\frac{|\iso(q)|}{2(n+1)!}$ & 
            $\frac{|\iso(q)|}{2(n+1)!}$ & 
            $\frac{|\iso(q)|}{2(n+1)!}$\\ \hline
            D & 
            $\frac{|\iso(q)|}{(n+1)!}$ & 
            $\frac{|\iso(q)|}{(n+1)!}$ & 
            $\frac{|\iso(q)|}{(n+1)!}$\\\hline 
            E & 
            $\frac{|\iso(q)|}{(n+2)!}$ & 
            $\frac{|\iso(q)|}{(n+2)!}$ & 
            $\frac{|\iso(q)|}{(n+2)!}$
        \end{tabular}
    \caption{Number of maximum Cliques of $\Gcal_{q,a}$ for finite fields of characteristic different from $2$.}
    \label{tbl:numbersofcliques:char:odd}
\end{table}
}

{\renewcommand{\arraystretch}{1.5}
\begin{table}[!ht]
            \begin{tabular}{c||c|c}
                case / isometry & $q$ hyperbolic & $q$ not hyperbolic\\\hline\hline
                A & $\frac{|\iso(q)|}{\alpha(n-1)!}$ & $\frac{|\iso(q)|}{\alpha(n-1)!}$ \\\hline
                B & $\frac{|\iso(q)|}{\alpha\cdot n!} + \frac{|\iso(q)|}{2\cdot n!}$ & $\frac{|\iso(q)|}{\alpha\cdot n!} + \frac{|\iso(q)|}{2\cdot n!}$\\\hline 
                C & $\frac{|\iso(q)|}{(n+1)!}$ & 
            $\frac{|\iso(q)|}{(n+1)!}$ \\\hline
            D & $\frac{|\iso(q)|}{(n+2)!}$ & 
            $\frac{|\iso(q)|}{(n+2)!}$ 
            \end{tabular}
    \caption{Number of maximum Cliques of $\Gcal_{q,a}$ for finite fields of characteristic $2$.}
    \label{tbl:numbersofcliques:char:2}
\end{table}
}

Lastly, in good fashion of quadratic form theory, we find a local-global principle for clique numbers over $\Q$.
\begin{mtheorem}[Local-global principle for clique numbers over $\Q$]
    Let $q$ be a quadratic form over $\Q$ of dimension $n$ and $a\in \Q^*$. 
    We have
    \begin{align*}
        d:=&\max\{k\in\N_0\mid \gamma_{k, a}\subseteq q\}\\
        =&\max\{k\in\N_0\mid \min\{i_W((q\perp -\gamma_{k,q})_p)\mid p\text{ is a prime or }\infty\}\geq k\}.
    \end{align*}
    In particular, we have $\omega(\Gcal_{q, a})=d+1$.
\end{mtheorem}

The proof of the theorem is given by Theorem~\ref{thm:local:global}.

\subsection{Structure of the paper}

In Section~\ref{Sec:IntroToInvariants} we provide some basic results about representation graphs, discussing in particular how isometry / similarity of quadratic forms and isomorphy of associated representation graphs are related. In particular we can bound the number of ``different'' representation graphs that can occur for any given quadratic form.

In Section~\ref{sec:structural:properties} we introduce a reduction of the representation graph of a quadratic form $q$ and a ring element $a$ to a smaller graph defined only on the elements that represent $a$ for $q$. The main use of this is not that these graphs are smaller per se, but rather that this linearizes the problem, see in particular Lemma~\ref{lem:LinCombOfCliqueInRed}. Building on this, we introduce a test form $\gamma_{a,n}$ and relate the clique number to the maximal dimension $n$ for which $q$ contains a subform isometric to $\gamma_{a,n}$. This allows us to employ algebraic tools to fruition.

In Section~\ref{sec:local:rings} we build upon the results of the preceding section. By purely algebraic considerations we are able to explicitly compute the clique number in various cases. Furthermore we can compute the number of maximum cliques over finite fields by means of the orbit-counting-theorem.

In Section~\ref{sec:characteristic:zero} we show how to recover the results of Chilakamarri \cite{Chilakamarri_1988} using the techniques of the preceding sections. 
As a byproduct we see that they hold in some more generality, namely not only over the rationals, but over any field of characterisic zero. 
In addition we compute the clique number for arbitrary quadratic forms over the real numbers. 
Lastly, we consider the relation of the clique number of forms over the rationals and their localizations over $p$-adic numbers.
We obtain a local-global principle that is applicable to arbitrary non-degenerate rational forms.

\section{Graph Properties of Representation Graphs}\label{Sec:IntroToInvariants}

We now recall some standard notation on graphs and provide some basic properties of representation graphs of quadratic forms.

Let $G = (V,E)$ be a graph, where $V$ is the \emph{vertex set} and $E \subseteq \binom{V}{2}$ the \emph{edge set} of $G$. We say that two vertices $u,v \in V$ are \emph{adjacent} in $G$ if $\{u,v\} \in E$. For a set $U \subseteq V$ the \emph{induced subgraph} is the graph $G[U] = (U,E[U])$, where $E[U] = \{ \{u,v\} \in E \ :\ u,v \in U \}$. 

A graph $G$ is \emph{complete} if $E = \binom{V}{2}$, i.e. if any two vertices are adjacent. 

A \emph{clique} $C \subseteq V$ is a vertex set such that the induced subgraph $G[C]$ is complete, i.e. every two vertices $u,v$ of $U$ are adjacent. 
The \emph{clique  number} $\omega(G)$ of $G$ is the size of a largest clique in $G$, a \emph{maximum Clique}. A clique is called \emph{maximal} if it is
not contained in any clique of larger size, note however, that a maximal clique is not necessarily maximum.

The group $\aut(G) = \{ \sigma \in \sym(V) \ :\ \{\sigma(u),\sigma(v)\} \in E \Leftrightarrow \{u,v\} \in E \}$ is the \emph{automorphism group} of the graph $G$. Graphs $G = (V,E),G' = (V',E')$ are \emph{isomorphic} if and only if there exists a bijective map $\rho: V \rightarrow V'$ such that $\{\rho(u),\rho(v)\} \in E' \Leftrightarrow \{u,v\} \in E$.

Since automorphisms (isomorphisms) of graphs preserve adjacency, it is immediately clear if $\sigma$ is any automorphism (isomorphism) a set $U \subseteq V$ is a clique if and only if $\sigma(U)$ is. 

In the context of representation graphs it is natural to ask how isomorphy of graphs relates to isometry of quadratic forms. For this let $A$ be a ring, $V$ an $A$-module, and $q:V \to A$ a quadratic form. Let $\og(q)$ be the \emph{orthogonal group} of isometries of $q$, that is maps
\[
    \varphi:\ V \rightarrow V,\ q(\varphi(v))=q(v).
\]
We write 
\begin{equation} \label{eq:affine:isometries}
    \iso(q) = V \rtimes \og(q)
\end{equation}
for the semi-direct product of the group of translations of the module $V$ and the orthogonal group $\og(q)$. We will sometimes refer to $\iso(q)$ as the affine isometry group of $q$ and its elements as affine isometries.

Let $\varphi$ be any (possibly affine) isometry of $q$. Then 
\[
    x \sim y \Longleftrightarrow \varphi(x) \sim \varphi(y).
\] 
Therefore we obtain an embedding of (affine) isometries of $q$ into the automorphism group of $\Gcal_{q,a}$, this already provides us with a good number of graph automorphism.

In particular, the group of translations $\tau_v:V \rightarrow V; x \mapsto v+x$ by module elements $v \in V$ acts transitively on $V$, and therefore the graph $\mathcal G_{d, a}=(V,E)$ is regular, i.e. all vertices have the same degree.

For further background on the theory of quadratic forms the reader is referred to \cite{MR2104929}.

In addition isometry of quadratic forms extends to isomorphy of graphs.
\begin{proposition} \label{IsometryAndSimilarityCalculationRules}
	Let $q, q'$ be quadratic forms on $A$-modules $V, V'$ and let $a, \alpha \in A$ with $\alpha$ not a zero-divisor. 
	\begin{enumerate}[label=(\alph*)]
		\item \label{IsometryAndSimilarityCalculationRulesA} If $q'$ is another quadratic form with $q\cong q'$, we then have $\mathcal G_{q,a}\cong\mathcal G_{q', a}$.
		\item\label{IsometryAndSimilarityCalculationRulesB} For every $\alpha\in A$ that is not a zero-divisor, we have $\mathcal G_{\alpha q, \alpha a}=\mathcal G_{q, a}$.
		\item\label{IsometryAndSimilarityCalculationRulesC} For every unit $u\in A^\ast$, we have $\mathcal G_{q, u^2a}\cong\mathcal G_{q, a}$.
	\end{enumerate}
\end{proposition}
\begin{proof}
	\begin{enumerate}[label=(\alph*)]
		\item This is a consequence of the discussion above.		
		\item Since $\alpha$ is not a zero divisor we have the equivalence $$(\alpha q)(x) = \alpha a \Longleftrightarrow q(x) = a$$ proving the claim.
        \item If $u$ is a unit $q \cong u^{-2} q$ and the result follows from the above.
	\end{enumerate}
\end{proof}

We briefly note that part $(a)$ of the above proposition is not reversible, isomorphism of representation graphs does not imply isometry of the involved quadratic forms.
We give some examples of different nature illustrating this.

\begin{example}
	\begin{enumerate}[label=(\alph*)] 
		\item let $A=\mathbb F_3=V$ and $q=\qf1, q'=\qf2, a= 0$. As both quadratic forms are anisotropic, the graphs $\mathcal G_{q, 0}$ and $\mathcal G_{q', 0}$ have empty edge sets and are thus isomorphic.
		\item As another example, we can consider $q=\qf1, q'=\qf2$ over $A=\Z/9\Z$ with $a=6$. Then, both graphs $\mathcal G_{q, a}, \mathcal G_{q', a}$ consist of 9 isolated points.
		\item For $A=\Z/6\Z = M, q=\qf1, q'=\qf 5 = 5q, a=3$, we have $\mathcal G_{q, 3}=\mathcal G_{5q, 5\cdot 3}=\mathcal G_{q', 3}$ since we have $5\cdot 3=3$ in $\Z/6\Z$. Further, $q$ and $q'$ are not isometric as $q$ does not represent $5$, but $q'$ certainly does.
	\end{enumerate}
\end{example}

It is not hard to see that representation graphs are Cayley graphs for the group $V$ and symmetric set $S_a = \{ x \in V \ :\ q(x) = a \}$, but we do not need this fact in the subsequent discussion.

We close this section by determining an upper bound for the number of non-isomorphic representation graphs for a given quadratic form using the group of similarity factors.

\begin{defrem}
	\begin{enumerate}[label=(\alph*)]
		\item A unit $u\in A^\ast$ is called a \textit{similarity factor} for the quadratic form $q$ if we have $uq\cong q$. The set of all similarity factors forms a subgroup of $A^\ast$, denoted by $G_A(q)$.
		\item A quadratic form is called \textit{round}, if we have $G_A(q)=D_A(q)\cap A^\ast$.
		\item We have an equivalence relation on $A$ defined by $a\sim b$ if there exists some $u\in G_A(q)$ with $b=ua$. 
		In particular, the set 
		$$A/G_A(q):=\{aG_A(q)\mid a\in A\}$$
		is well defined.
	\end{enumerate}
\end{defrem}

\begin{proposition}
	Let $q$ be a quadratic form on an $A$-module $M$. We then have at most $|A/G_A(q)|$ different non-isomorphic graphs of the form $\mathcal G_{q, a}$ with $a\in A$.
\end{proposition}
\begin{proof}
	Let $a,b\in A$ be given such that there is some $u\in G_A(q)$ with $b=ua$. Since $u$ is a unit and thus not a zero divisor, we have
	$$\mathcal G_{q, a}\cong\mathcal G_{uq, ua}\cong \mathcal G_{q, b}.$$
	by Proposition~\ref{IsometryAndSimilarityCalculationRules} \ref{IsometryAndSimilarityCalculationRulesA} and \ref{IsometryAndSimilarityCalculationRulesB} and the claim follows.
\end{proof}

\begin{corollary}
	Let $F$ be a field and $q$ a round form over $F$. We then have at most 3 non-isomorphic graphs of the form $\mathcal G_{q, a}$. If $q$ is further universal, we have at most 2 non-isomorphic such graphs.
\end{corollary}

\section{Structural Properties of Cliques in $\Gcal_{q,a}$} \label{sec:structural:properties}

\subsection{The Reduced Graph}

In general, computing the clique number is hard. For a representation graph $G$, we may just consider the following subgraph of $G$ who contains all information that is needed to compute the clique number.

\begin{definition}
    Let $G = \Gcal_{q,a} = (V,E)$ be a representation graph of a quadratic form $q$. The graph $\red G =\Gcalred_{q,a}= (\red V,\red E) = G[\red V]$ with
    \[
        \red V = \{ v \in V \ :\ q(v) = a \} \quad\text{and}\quad \red E = E[\red V]
    \]
    is the \emph{reduced} representation graph of $q$ \emph{with respect to} $a$.
\end{definition}

\begin{proposition}\label{CliqueNrWholeGraphReducedGraph}
    Let $G = \Gcal_{q,a} = (V,E)$ be a representation graph of a quadratic form $q$. 
    Then $$ \omega(G) = \begin{cases}
        \omega(\red G) + 1,&\text{ if }a\neq 0,\\
        \omega(\red G),&\text{ if }a= 0
    \end{cases}.$$
\end{proposition}
\begin{proof}
    For any $v \in V$ the translation $\tau_{-v}$ by $-v$ is an isometry of $(V, d_q)$ and thus an automorphism of $G$ as mentioned in Section~\ref{Sec:IntroToInvariants}.
    If $C$ is a clique in $G$ and $v\in C$, $C_0:=\tau_{-v}(C)$ is therefore a clique in $G$ of the same size that contains $0$.
    We further have $C_0 \setminus \{0\} \subseteq \red V$, since $q(w) = q(w-0) = a$ for every non-zero $w \in C_0$ by adjacency of $w$ and $0$. Thus $C_0 \setminus \{0\}$ is a clique in $\red G$ and if $a=0$, $C_0$ is a clique in $\red G$. This shows the upper bound for $\omega(G)$.

    On the other hand let $C' \subseteq \red V$ be a clique in $\red G$. 
    Then, by definition, $C' \cup \{0\}$ is a clique in $G$, showing the other inequality.
\end{proof}

The next result shows how we can translate cliques in $\red V\cup\{0\}$.

\begin{lemma}\label{TranslationOfCliquesInVred}
    Let $C$ is a clique in $\red V\cup\{0\}$ containing $0$ and $v\in V$. 
    Then $C-v=\{x-v\mid x\in C\}$ is a clique in $\red V\cup\{0\}$ containing $0$ if and only if $v\in C$.
\end{lemma}
\begin{proof}
    It is clear that we have $0\in C-v$ only if $v\in C$.
    So let now $x,y,v\in C$ with $x\neq y$.
    If $x=v$, we clearly have $x-v=0\in \red V\cup\{0\}$. 
    If $x\neq v$, we have $q(x-v)=a$ since $x,v$ are different elements of the clique $C$ and thus adjacent. 
    Finally we have
    $$q\left((x-v)-(y-v)\right)=q(x-y)=a$$
    with the same argumentation as above.
\end{proof}

For the rest of this section, we will always assume that $A$ is a ring, $a\in A$ and $q$ is a quadratic form on an $A$-module $V$. 
The associated bilinear form is denoted by $b:=b_q$.
Let $G=\Gcal_{q,a} = (V,E)$ be the representation graph of $q$ with respect to $a$ and $\red G = (\red V,\red E) = G[\red V]$ be the reduced representation graph. 
We assume $\red G$ to be non-empty. 
This is clearly fulfilled for finite fields and forms of dimension at least $2$. 
Notice that we do not restrict our rings to be finite since most of the arguments also hold over arbitrary rings resulting in infinite graphs. 

\begin{lemma}\label{BilinearFormCliqueReducedGraph}
    Let $x\neq y\in V(\red G)$. They are adjacent if and only if we have $b(x,y)=a$. 
\end{lemma}
\begin{proof}
    The two vertices $x,y$ are adjacent in $\red G$ if and only if we have $q(x-y) = a$. 
    We further have
    \begin{align*}
        q(x-y)=q(x)+q(y)-b(x,y)=2a-b(x,y)
    \end{align*}
    which is equivalent to $b(x,y)=2a-q(x-y)$.
    Putting the above equivalences together, we see that $x,y$ are adjacent if and only if $a=b(x,y)$.
\end{proof}

\begin{lemma}\label{lem:LinCombOfCliqueInRed}
    Let $x_1,\ldots, x_n\in \red V$ be a clique, let $\lambda_1,\ldots, \lambda_n\in A$ and $x=\sum\limits_{i=1}^n \lambda_ix_i$. Then
    $$q(x)=a\sum\limits_{1\leq i\leq j\leq n}\lambda_i\lambda_j.$$
\end{lemma}
\begin{proof}
    We have $q(x_i)=a$ for all $i\in\{1,\ldots, n\}$ and $b(x_i, x_j)=a$ for all $i,j\in\{1,\ldots, n\}$ with $i\neq j$ by definition of $\red V$ and Lemma~\ref{BilinearFormCliqueReducedGraph}. 
    For $n=1$, we have 
    \begin{align*}
        q(\lambda_1x_1)=\lambda_1^2q(x_1)=a\lambda_1^2
    \end{align*}
    as desired.
    For $n\geq2$, we have obtain
    \begin{align*}
        q(x)&=q\left(\sum\limits_{i=1}^{n-1}\lambda_i x_i +\lambda_n x_n\right)\\
        &=q\left(\sum\limits_{i=1}^{n-1}\lambda_i x_i\right)+q(\lambda_nx_n)+b\left(\sum\limits_{i=1}^{n-1}\lambda_i x_i, \lambda_nx_n\right)\\
        &=q\left(\sum\limits_{i=1}^{n-1}\lambda_i x_i\right) + \lambda_n^2\underbrace{q(x_n)}_{=a} + \lambda_n\sum\limits_{i=1}^{n-1}\lambda_i\underbrace{b(x_i, x_n)}_{=a}\\
        &=q\left(\sum\limits_{i=1}^{n-1}\lambda_i x_i\right) + a\sum\limits_{i=1}^n\lambda_i\lambda_n
    \end{align*}
    and the claim follows by induction.
\end{proof}

\subsection{Isotropic Cliques}

The above result directly shows how we can determine large cliques for $a=0$, i.e. cliques consisting of isotropic vectors.

\begin{proposition}
    Let $a=0$ and $x_1,\ldots, x_n\in C\subseteq \red V$ be elements of some clique $C$ in $\red G$. Then $W:=\spn(x_1,\ldots, x_n)$ is a totally isotropic subspace and thus a clique in $\red G$. 
\end{proposition}
\begin{proof}
    Putting $a=0$ in Lemma~\ref{lem:LinCombOfCliqueInRed}, we readily see that $\spn(x_1,\ldots, x_n)$ is a totally isotropic subspace.

    Since totally isotropic subspaces are closed under subtraction, they are obviously cliques in $\red G$.
    %\begin{align*}
        %2q\left(\sum\limits_{i=1}^n\lambda_ix_i\right)
        %&=b\left(\sum\limits_{i=1}^n\lambda_ix_i, \sum\limits_{i=1}^n\lambda_ix_i\right)\\
    %    &=\sum\limits_{i, j=1}^n\lambda_i\lambda_jb(x_i, x_j)=0,
    %\end{align*}
    %which by induction implies $q\left(\sum\limits_{i=1}^n\lambda_ix_i\right)=0$. 
    %We thus have $W\subseteq \red V$ or in other words, $W$ is a totally isotropic subspace and thus a clique as pointed out above. 
\end{proof}

Over local rings, the dimension of a maximal totally isotropic subspace (short: m.t.i.s.) can be computed using the Witt index, see \cite[Chapter I, (3.6) Theorem, (4.6) Theorem]{MR0491773}. Using Proposition~\ref{CliqueNrWholeGraphReducedGraph}, we thus obtain the following:

\begin{corollary}
    Let $A$ be a finite local ring, $q$ be a non-degenerate form quadratic form. Let further $i= i_W(q)$ be the Witt index of $q$. We have $\omega(\Gcal) = |A|^i$. 
\end{corollary}

\begin{remark}
	The number of maximal totally isotropic subspaces is well known, see \cite[Theorem 3.2]{MR133725} for the case of a finite field and TODO for the case of a finite local ring. Let $q$ be a quadratic form of dimension $n$ over a finite field with $f$ elements. Recall that the Witt index $i:=i(q)$ of $q$ is given by $\left\lfloor \frac n2\right\rfloor$ if $n=2k+1$ is odd and by $\frac n2$ or $\frac n2-1$ if $n=2k$ is even, depending on whether $q$ is hyperbolic or not.
	The number of maximal totally isotropic subspaces is then given by 
	\begin{align*}
		\prod\limits_{r=0}^{k-1}\frac{f^{2(r+1)}-1}{f^{-r}-1},&\text{ if }n=2k+1\\
		\prod\limits_{r=0}^{k-1}\frac{f^{2(r+1)}-f^{r}+f^{r+1}-1}{f^{k-r}-1},&\text{ if }n=2k, q\text{ is hyperbolic}\\
		\prod\limits_{r=0}^{k-1}\frac{f^{2(r+2)}+f^{r+1}-f^{r+2}-1}{f^{k-r}-1},&\text{ if }n=2k, q\text{ is not hyperbolic}\\
	\end{align*}
	For finite local rings the situation is more involved. A treatment can be found in \cite{zimmermann:2017:phd}. Here several types of m.t.i.s. (cf. \cite[Theorem 2.4.4]{zimmermann:2017:phd}) appear and two m.t.i.s. are in the same orbit of the orthogonal group of the quadratic form if and only if they are of the same type. The number of m.t.i.s. of a fixed type is given in \cite[Proposition 2.5.5]{zimmermann:2017:phd}). Note that these results directly generalize to arbitrary finite rings (cf. \cite[Proposition 2.5.1]{zimmermann:2017:phd}).
\end{remark}

\subsection{Non-Isotropic Cliques}

We now turn our attention to the case of cliques where the difference of the vectors are non-isotropic vectors. 

We denote the identity matrix and the matrix whose entries are all $1$ of size $n\times n$ by $I_n$ respectively $J_n$. 
We will further denote the all ones vector of length $n$ by $\mathbf 1_n$.
If the integer $n$ is clear from the context, we will drop the subscript to simplify our notation.

\begin{lemma}\label{LemUpperBoundClique}
    Let $a\in A$ be not a zero divisor, $n\geq 2$ and $x_1,\ldots, x_n\in C\subseteq\red V$ be pairwise distinct elements of a clique $C\subseteq \red V$ and let $W:=\spn(x_1,\ldots, x_n)$.
    Let $k=n+2$ if $2\in A^\ast$ and $k=\frac{n+2}2$ otherwise.
    If $k\notin\Z$ or $\cha A\nmid k$, we have
    \begin{align*}
        W\cap C = \{x_1,\ldots,x_n\}
    \end{align*}
    and otherwise, we have 
    \begin{align*}
        W\cap C\subseteq\{x_1,\ldots,x_n\} \cup \{-\sum\limits_{i=1}^nx_i\}.
    \end{align*}
    Moreover, we can add $-\sum\limits_{i=1}^nx_i$ in the latter case to $C$ and still obtain a clique.
\end{lemma}
\begin{proof}
    Assume there exists $x\in (W\cap C)\setminus\{x_1,\ldots, x_n\}$. Then, there are $\lambda_1,\ldots, \lambda_n\in A$ with $x=\sum\limits_{i=1}^n\lambda_ix_i$. As we have $x\in C \subseteq \red V$ and by using Lemma~\ref{lem:LinCombOfCliqueInRed} we obtain
    \begin{align*}
        a=q(x)=a\sum\limits_{1\leq i\leq j\leq n}\lambda_i\lambda_j
    \end{align*}
    which implies  
    \begin{align}\label{eq:FirstCondition}
        1=\sum\limits_{1\leq i\leq j\leq n}\lambda_i\lambda_j
    \end{align}
    because $a$ is not a zero divisor.\\
    Since we have $x\in C\setminus\{x_1,\ldots, x_n\}$, we further have $q(x-x_k)=a$ for all $k\in\{1,\ldots, n\}$. We have
    \begin{align*}
        x-x_k=\sum\limits_{i=1}^n\mu_ix_i\text{ where }\mu_i=\begin{cases}
            \lambda_i,&\text{if }i\neq k,\\
            \lambda_i-1,&\text{if } i = k.
        \end{cases}
    \end{align*}
    As we have $x-x_k\in \red V$ we can argue as above to obtain 
    \begin{align*}
        1&=\sum\limits_{1\leq i\leq j\leq n}\mu_i\mu_j=\sum\limits_{1\leq i\leq j\leq n}\lambda_i\lambda_j-\sum\limits_{\substack{i=1\\i\neq k}}^n\lambda_i+1-2\lambda_k
    \end{align*}
    which is, using \eqref{eq:FirstCondition}, equivalent to $2\lambda_k+\sum\limits_{\substack{i=1\\i\neq k}}^n \lambda_i=1$.
    This holds for all $k\in\{1,\ldots, n\}$, so the vector $\lambda = (\lambda_1,\ldots, \lambda_n)^\top$ satisfies the equality 
    \begin{align}\label{SysLinEq}
        (I + J)\lambda = \mathbf 1,
    \end{align}
    Since we have $J\lambda\in\spn(\mathbf 1)$, equation \eqref{SysLinEq} implies $\lambda\in \spn(\mathbf 1)$. Thus there exists $\alpha \in A$ with $\lambda = \alpha\mathbf 1$ and we obtain
    \begin{align*}
        \mathbf 1 = (I+J) \alpha \mathbf 1= \alpha I\mathbf 1+ \alpha J\mathbf 1= \alpha \mathbf 1+n \alpha \mathbf 1= \alpha(1+n)\mathbf 1.
    \end{align*}
    This is equivalent to $(1+n)\alpha = 1$.
    This is thus solvable if and only if $1+n\in A^\ast$ and the unique solution is given by $\alpha = (n+1)^{-1}$.\\
    To determine when we can add $x:=\frac{1}{n+1}\sum\limits_{i=1}^nx_i$ to $C$ and still obtain a clique in $\red G$, we first remark that by construction of the above system of linear equations, we have $q(x-x_k)= a$ for all $k\in \{1,\ldots, n\}$ if we have $q(x)=a$. 
    
    We thus need only to determine under which circumstances we have $q(x)=a$.
    By Lemma~\ref{lem:LinCombOfCliqueInRed} and using that $a$ is not a zero divisor, this is the case if and only if we have
    \begin{align}\label{eq:AddingToAClique}
        (n+1)^2=\sum\limits_{1\leq i\leq j\leq n}1.
    \end{align}
    If 2 is invertible, the right hand sum is given by $\frac{n(n+1)}2$ and the equality becomes $n+2=0$, i.e. $\cha A\mid n+2$, which implies $n+1=-1=(n+1)^{-1}$.\\
    If 2 is not invertible, $n$ has to be even since otherwise, $1+n$ would not be a unit.
    In particular, we have $\frac n2\in\Z$ and \eqref{eq:AddingToAClique} is equivalent to
    \begin{align*}
        n+1=\frac n2\iff \frac n2+1=0,
    \end{align*}
    i.e. $\cha A\mid \frac n2+1$ which again implies $n+1=-1=(n+1)^{-1}$.
\end{proof}

\begin{remark} \label{rem:clique:independent}
    Note that in the situation of Lemma~\ref{LemUpperBoundClique} any clique $\{x_1,\ldots,x_n\}$ contains at least $n-1$ linearly independent vectors. To see this choose a subset of $\{x_1,\ldots,x_n\}$ which is a basis of the space $W=\spn (x_1,\ldots,x_n)$ and applying the lemma to it instead.
\end{remark}

While the above results yield upper bounds for the maximum size of a clique in $\red G$, we now turn to the construction of biggest possible cliques and thus to lower bounds.

As we have already seen, quadratic forms of the following type play a crucial role when determining the clique number.

\begin{defrem}\label{MaxCliqueForm}
    We write $\gamma_{a,n}$ for the $n$-dimensional quadratic form on $A^n$ with $\gamma_{a,n}(e_i)=a$ for all $i\in\{1,\ldots, n\}$ and whose associated bilinear form has $a(I_n+J_n)$ as a Gram matrix for the canonical basis $(e_1,\ldots, e_n)$. 
    
    The determinant of the Gram matrix is given by $(n+1)\cdot a^n$. 
    In particular, $\gamma_{a,n}$ is non-degenerate if and only if $a\in A^\ast$ and $(n+1)\in A^\ast$ and we have $\gamma_{a,n}\cong a\gamma_{1, n}$.
\end{defrem}
\begin{proof}
    Replacing the last row with the sum of all rows and then subtracting the last column from all others will yield an upper triangle matrix with diagonal entries $a, \ldots, a, (n+1)a$, so the claim follows.
\end{proof}

Note that $\gamma_{a,n}$ is no new object. Readers familiar with euclidean lattices will recognize $\gamma_{1, n}$ as the quadratic form associated to the $A_n$ root lattice, see e.g. \cite[Theorem 1.2]{MR2977354}.

Since the determinant of the associated bilinear form does not contain much information in characteristic 2, we determine the Arf invariant $\Delta(\gamma_{a_n})$ to be able to identify the isometry type of this form in some special cases.

\begin{lemma}\label{lem:ArfQ_an}
    Let $F$ be a field of characteristic 2, $a\in F^\ast$ and $n\in\N$ with $n\equiv 0\mod 2$. We then have
    $$\Delta(\gamma_{a,n})=\begin{cases}
        1,&\text{ if }n\equiv 2, 4\mod 8\\
        0,&\text{ if }n\equiv 0, 6\mod 8.
    \end{cases}.$$
\end{lemma}
\begin{proof}
    For $n=2$, the form is given by $\gamma_{a, 2}\cong [a, a^{-1}]$ and has thus Arf invariant 1. 
    Let now $n$ be an even integer with $n\geq4$. 
    We consider the vectors 
    $$f_1:=e_{n-1}+\sum\limits_{k=1}^{n-2}e_k, ~~~ f_2:=e_{n}+\sum\limits_{k=1}^{n-2}e_k$$ 
    and the subspaces
    $$V:=\spn(e_1,\ldots, e_{n-2})\text{ and }W:=\spn(f_1, f_2).$$
    Clearly, $q|_V\cong \gamma_{a,n-2}$ is a regular space and it is easy to see that $V^\perp=W$.
    We have 
    \begin{align*}
        b(f_1, f_2)&=b(e_1+\ldots + e_{n-1}, e_1+\ldots+ e_{n-2} + e_n)\\
        &=b(e_1+\ldots+e_{n-1}, e_n) + b(e_1+\ldots + e_{n-1}, e_1+\ldots+e_{n-2})\\
        &=\sum\limits_{k=1}^{n-1}\underbrace{b(e_k, e_n)}_{=a} + \sum\limits_{j=1}^{n-2}\underbrace{\sum\limits_{k=1}^{n-1}b(e_k, e_j)}_{=(n-2)\cdot a=0}\\
        &=(n-1)\cdot a = a.
    \end{align*}
    Further, by Lemma~\ref{lem:LinCombOfCliqueInRed}, we have
    \begin{align*}
        q(f_1)= q(f_2)=\begin{cases}
            a, &\text{ if }n\equiv 2\mod 4\\
            0, &\text{ if }n\equiv 0\mod 4.
        \end{cases}
    \end{align*}
    In the first case, we have $q|_W\cong[a, a^{-1}]$ with Arf-invariant 1 and in the other case, $q|_W$ is hyperbolic with Arf-invariant 0.
    The claim now follows from the equation $\Delta(q)=\Delta(q|_V)+\Delta(q|_W)$ by induction.
\end{proof}

Let $(V,q)$ be a regular quadratic form and let $C$ be a clique for $\Gcalred_{q,a}$. 
Let $W = \spn(C)$ and $m:=\dim(W)$. 
Then by Lemma~\ref{BilinearFormCliqueReducedGraph} the quadratic space $(W,q_{\mid W})$ is isomorphic to $\gamma_{a, m}$.

Assume that $C$ is a maximal clique for $\Gcalred_{q,a}$. Then 
\[
   \omega(\Gcal_{q,a}) = \begin{cases}
        m + 2, & \text{if }k\notin\Z\text{ or }\cha A\mid k, \\
        m + 1, & \text{otherwise.}
    \end{cases}
\]
where $k$ is defined as above by $k=m+2$ if $2\in A^\ast$ and $k=\frac{m+2}2$ otherwise.

Thus the problem of determining $\omega(\Gcal_{q,a})$ reduces to determining the dimension of a largest (not necessarily non-degenerate) subform $\gamma_{a,m}$ of $q$.
The rest of this subsection will therefore present techniques to find suitable subforms.
We start with the following general criterion for fields of characteristic $\neq 2$ which is a special case of \cite[(3.10)]{MR2788987}. Since the reference is only available in German and we have not found it elsewhere in the literature we include a proof.

\begin{lemma}\label{SubformWittIndex}
    Let $\varphi,\psi$ be quadratic forms over a field $F$ of characteristic $\neq 2$. Then $\psi$ is a subform of $\varphi$ if and only if $i_W(\varphi\perp-\psi)\geq\dim(\psi)$.
\end{lemma}
\begin{proof}
    If $\psi$ is a subform of $\varphi$, i.e. we have $\varphi\cong\psi\perp\varphi'$, we then have 
    \begin{align*}
        i_W(\varphi\perp-\psi)&=i_W(\psi\perp\varphi'\perp-\psi)\\
        &=i_W(\varphi'\perp\dim(\psi)\times\mathbb H)\geq \dim(\psi).
    \end{align*}
    For the other implication, we use induction on $\dim(\psi)$. If $\dim(\psi)=1$, we have $\psi\cong\qf{a}$ for some $a\in F^\ast$. Since $\varphi\perp\qf{-a}$ is isotropic, either $\varphi$ is isotropic and thus contains $\mathbb H\cong\qf{a,-a}$ as a subform or represents $a$ and has thus $\qf{a}$ as a subform. \\
    For the case $\dim(\psi)\geq 2$, we write $\psi\cong\psi'\perp\qf{a}$ for some quadratic form $\psi'$ and some $a\in F^\ast$. Since we have $i_W(\varphi\perp-\psi)\geq \dim(\psi)$, we have $i_W(\varphi\perp-\psi')\geq \dim(\psi')$ and thus $\psi'\subseteq\varphi$ by induction. We write $\varphi\cong\psi'\perp\varphi'$. Now $\varphi\perp-\psi=\varphi'\perp\qf a\perp(\psi\perp-\psi)$ and thus, $\varphi'\perp\qf a$ is isotropic. By the case of dimension $1$, we see that $\qf{a}\subseteq\varphi'$ and thus $\psi\subseteq\varphi$. 
\end{proof}

The following result will help us to decide whether we can find a subform of the shape $\gamma_{a,n-1}$ in an $n$-dimensional vector space when $\gamma_{a,n-1}$ is not regular.

\begin{lemma}\label{DetWhenNotRegularCodim1}
    Let $a\in A$, $v\in A^{n-1}, x\in A$ and let $\cha A\mid n$. The matrix
    $$M:=\left(\begin{array}{ccc|c}
        & & & \\
        & a(I_{n-1}+J_{n-1}) & & v\\
        & & &\\\hline
        & v^T & & x
    \end{array}\right)$$
    has determinant $$\det(M)=-\left(\sum\limits_{i=1}^{n-1}v_i\right)^2(n-1)a^{n-2}.$$
\end{lemma}
\begin{proof}
    Using $\cha A\mid n$, replacing the first column of $M$ by the sum of the columns $1,\ldots, n-1$ of $M$ and then replacing the first row by the sum of the rows $1,\ldots, n-1$ yields the matrix
    \begin{align*}
        \left(\begin{array}{ccccc|c}
           0 & \ldots & \ldots & \ldots & 0 & \sum\limits_{i=1}^{n-1}v_i \\
           \vdots & 2a & a & \ldots & a & v_2 \\
           \vdots & a & \ddots & \ddots & \vdots & \vdots\\
           \vdots & \vdots & \ddots & \ddots & a & \\
           0 & a & \ldots & a & 2a & v_{n-1} \\\hline
           \sum\limits_{i=1}^{n-1}v_i & v_2 & \ldots & \ldots & v_{n-1} & x
        \end{array}\right).
    \end{align*}
    Expanding by the first row and then by the first column (or vice versa), we readily see that the determinant of this matrix is given by
    $$(-1)^{(n+1) + ((n-1) + 1)}\left(\sum\limits_{i=1}^{n-1} v_i\right)^2\cdot \det\left(a(I_{n-2} + J_{n-2})\right).$$
    Since the transformed matrix has the same determinant as $M$, we finally obtain
    $$\det(M)= -\left(\sum\limits_{i=1}^{n-1}v_i\right)^2(n-1)a^{n-2}$$
    by plugging in the computation of $\det\left(a(I_{n-2} + J_{n-2})\right)$ made above in~\ref{MaxCliqueForm}.

\end{proof}

\section{The Case of Finite Local Rings} \label{sec:local:rings}

\subsection{The clique numbers}

We now specialize our results to several important types of rings. We begin with the case of finite fields and finite local rings.

\begin{corollary}\label{FiniteFieldQ_n-1AsNonRegularSubform}
    Let $A$ be a finite field in which 2 is invertible and $q$ be a regular quadratic form such that the associated bilinear form $b:=b_q$ has Gram matrix 
    $$\left(\begin{array}{ccc|c}
        & & & \\
        & a(I_{n-1}+J_{n-1}) & & v\\
        & & &\\\hline
        & v^T & & x
    \end{array}\right)$$
    for some $a\in A, v\in A^{n-1}, x\in A$ and let $\cha A\mid n$. We then have an isometry $q\cong \gamma_{a,n}$.
\end{corollary}
\begin{proof}
    By \ref{MaxCliqueForm}, the Gram matrix of $b_{\gamma_{a,n}}$ has determinant $(n+1)a^n = a^n$ and thus equals $\det(b)$ up to a square by Lemma~\ref{DetWhenNotRegularCodim1}.
    The assertion now follows from the fact that a quadratic form over a finite field of characteristic not $2$ is determined up to isometry by its dimension and its determinant modulo squares. 
\end{proof}

As usual, the case of a local ring in which 2 is invertible can be reduced to the case of a field using the residue field.

\begin{corollary}\label{FiniteLocalRingQ_n-1AsNonRegularSubform}
    Let $A$ be a finite local ring in which 2 is invertible, but $n$ is not. Denote the maximal ideal of $A$ by $\mathfrak m$ and let $K:=A/\mathfrak m$ be the residue field. Let further $q$ be a regular quadratic form such that the associated bilinear form $b:=b_q$ has Gram matrix 
    $$\left(\begin{array}{ccc|c}
        & & & \\
        & a(I_{n-1}+J_{n-1}) & & v\\
        & & &\\\hline
        & v^T & & x
    \end{array}\right)$$
    for some $a\in A, v\in A^{n-1}, x\in A$. We then have an isometry $q\cong \gamma_{a,n}$.
\end{corollary}
\begin{proof}
    The map that sends a quadratic form respectively bilinear form over $A$ to the corresponding form over the residue field $K$ by reducing the coefficents, is an isomorphism of the (bilinear) Witt rings, see \cite[Chapter V, (1.5) Corollary]{MR0491773}. The claim now follows readily from Corollary~\ref{FiniteFieldQ_n-1AsNonRegularSubform} since $n$ not being invertible in $A$ is equivalent to $n$ being a zero-divisor in the finite ring $A$ which then is equivalent to $\cha K\mid n$.
\end{proof}

We now consider the case whether it is possible to have $\gamma_{a,n-1}$ as a (not necessary non-degenerate) subform of some form $q\not\cong \gamma_{a_n}$ of dimension $n$ over a field of characteristic 2.
This is the concept of \emph{dominating forms} introduced in \cite[Definition 3.4]{MR2058517}.
In our situation, this will be the special case of a \emph{nonsingular completion}.

\begin{lemma}\label{lem:q_n-1AsSubformChar2}
    Let $F$ be a finite field of characteristic 2, $n$ an even integer, $q$ an $n$-dimensional non-degenerate quadratic form over $F$ with $q\not\cong \gamma_{a,n}$ but such that $\gamma_{a,n-1}$ is dominated by $q$.
    Then $n\equiv 0\mod 4$ and we have $q\cong \gamma_{a,n-2}\perp[a, b]$ for some $b\in F$.
    Conversely, if $n\equiv 0\mod 4$, then $\gamma_{a,n-2}\perp[a,b]$ dominates $\gamma_{a,n-1}$ for all $b\in F$.
\end{lemma}
\begin{proof}
    We consider the bilinear form $b=b_q$ associated to $q$. 
    For a suitable basis $(x_1,\ldots, x_n)$, the Gram matrix is given by 
    $$\left(\begin{array}{ccc|c}
        & & & \\
        & a(I_{n-1}+J_{n-1}) & & v\\
        & & &\\\hline
        & v^T & & x
    \end{array}\right)$$
    for some $v\in F^{n-1}, x\in F$.
    By \ref{MaxCliqueForm}, there are $\lambda_1,\ldots, \lambda_{n-2}$ such that 
    $$a(I_{n-2}+J_{n-2})\cdot\begin{pmatrix}
        \lambda_1\\\vdots\\\lambda_{n-2}
    \end{pmatrix}=\begin{pmatrix}
        v_1\\\vdots\\v_{n-2}
    \end{pmatrix}.$$
    We define
    $$y_1:=\sum\limits_{i=1}^{n-1}x_i,~~~ y_2:=x_n+\sum\limits_{i=1}^{n-2}\lambda_ix_i.$$
    Then $(x_1,\ldots, x_{n-2}, y_1, y_2)$ is a basis of $F^n$ and the Gram matrix for this basis is given by
    \begin{align*}
        \left(\begin{array}{ccc|cc}
        & & & 0 & 0\\
        & a(I_{n-1}+J_{n-1}) & & \vdots & \vdots\\
        & & & 0 & 0 \\\hline
        0 & \cdots & 0 & 0 & b(y_1, y_2)\\
        0 & \cdots & 0 & b(y_1, y_2) & 0
    \end{array}\right)
    \end{align*}
    By Lemma~\ref{lem:LinCombOfCliqueInRed}, we have 
    $$q(f_1)=\begin{cases}
        0,&\text{ if }n\equiv 0\mod 4\\
        a,&\text{ if }n\equiv 2\mod 4.
    \end{cases}$$
    For $n\equiv 2\mod 4$, we already know we have $q\cong \gamma_{a,n-2}\perp[a, b]$ for some $b\in F$.
    Finally for $n\equiv 0\mod 4$, we have
    $$q\cong \gamma_{a,n-2}\perp\mathbb H\cong \gamma_{a,n}$$
    using Lemma~\ref{lem:ArfQ_an} and the fact that the isometry type of a quadratic form over a finite field of characteristic 2 is uniquely determined by its dimension and its Arf invariant.\\
    The opposite direction follows since we can reverse the above calculations.
\end{proof}

Before we state our main theorem, we would like to note that we have $\gamma_{a, 1}\cong\qf a$. The following theorem is a generalized version of the main results of M. Krebs to be found in \cite[Theorems 1.2, 1.3]{Krebs_2022}.

\begin{theorem}\label{CliqueNumbersTheorem}
    Let $A$ be a finite local ring whose residue field $K$ has odd characteristic and $q$ be an $n$-dimensional quadratic form over $A$ and $a\in A$ be not a zero divisor (and hence a unit). 
    We then have the following values for $\omega(\Gcal_{q,a})$:
    \begin{align*}
        \begin{tabular}{c|c|c|c}
            Case & isometry condition & characteristic condition &  $\omega(\Gcal_{q,a})$\\ \hline
            A & $q\not\cong \gamma_{a,n}$ & $\cha K\mid n$ & $n$ \\ \hline
            B & $q\not\cong \gamma_{a,n}$ & $\cha K\nmid n, n+1$ & $n$ \\ \hline
            C & $q\not\cong \gamma_{a,n}$ & $\cha K\mid n+1$ & $n+1$ \\ \hline
            D & $q\cong \gamma_{a,n}$ & $\cha K\nmid n+2$ & $n+1$ \\ \hline
            E & $q\cong \gamma_{a,n}$ & $\cha K\mid n+2$ & $n+2$
        \end{tabular}
    \end{align*}
\end{theorem}
\begin{proof}
    We use Proposition~\ref{CliqueNrWholeGraphReducedGraph} in all cases and will handle the cases $\dim(q)=1$ and $\dim(q)=2$ seperately. 
    Note that for $n=1$, only cases B, D, E and for $n=2$, only cases B, C, D can occur.
    If $\dim(q)=1$, the assertion follows from Lemma~\ref{LemUpperBoundClique}. 
    For the case $\dim(q)=2$, we first note that we can have $q\cong \gamma_{a, 2}$ only if $\cha K\neq 3$ since otherwise, $q$ would be not regular. 
    The case where we have this isometry is thus clear by Lemma~\ref{LemUpperBoundClique} as we do have $\cha K\nmid 4=\dim(q)+2$. 
    If $q\not\cong \gamma_{a, 2}$, $q$ still represents $a$ so that we have $\qf a=\gamma_{a, 1}\subseteq q$ and the claim now follows by the 1-dimensional case.\\
    We now consider the case $n\geq 3$. 
    If we have $q\cong \gamma_{a,n}$, the assertion follows as above from Lemma~\ref{LemUpperBoundClique}. 
    So let now $q\not\cong \gamma_{a,n}$. 
    If $n\in A^\ast$, i.e. $\cha K\nmid n$, the form $\gamma_{a,n-1}$ is regular. 
    It is further a subform of $q$ since $q$ and $\gamma_{a,n-1}\perp\qf{\det(\gamma_{a,n-1})\cdot \det (q)}$ are regular forms of the same dimension and with the same determinant up to multiplication by a square and therefore isometric. 
    This case is thus also clear by Lemma~\ref{LemUpperBoundClique}.
    In the remaining case, we know from Corollary~\ref{FiniteFieldQ_n-1AsNonRegularSubform} respectively \ref{FiniteLocalRingQ_n-1AsNonRegularSubform} that any form that has $\gamma_{a,n-1}$ as a subform is already isometric to $\gamma_{a,n}$ which we excluded for $q$. 
    Thus $q$ cannot have $\gamma_{a,n-1}$ as a subform, but similarly as before, we can find the regular form $\gamma_{a,n-2}$ as a subform of $q$. 
    This case now follows from the same arguments as before involving Lemma~\ref{LemUpperBoundClique}.
\end{proof}

The above proof gives information on the co-dimension the maximal test form in $q$.

\begin{corollary} \label{cor:co-dim}
    In the situation of the Theorem~\ref{CliqueNumbersTheorem}, $q$ contains a subform isometric to $\gamma_{a,k}$ with
    \[
        k = 
        \begin{cases}
            n-2 & \text{in case A}\\
            n-1 & \text{in cases B and C}\\
            n & \text{in cases D and E}.
        \end{cases}
    \]
\end{corollary}

As the reduction theory does not work when 2 is not invertible, we only solve the case of finite fields.

\begin{theorem}\label{CliqueNumbersTheoremChar2}
    Let $F$ be a finite field of characteristic 2 and $q$ be an $n$-dimensional quadratic form over $F$ and $a\in F^\ast$. 
    We then have the following values for $\omega(\Gcal_{q,a})$:
    \begin{align*}
        \begin{tabular}{c|c|c|c}
            Case & isometry condition & dimension condition &  $\omega(\Gcal_{q,a})$\\ \hline
            A & $q\not\cong \gamma_{a,n}$ & $n\equiv 2\mod 4$ & $n-1$ \\ \hline
            B & $q\not\cong \gamma_{a,n}$ & $n\equiv 0\mod 4$ & $n$ \\ \hline
            C & $q\cong \gamma_{a,n}$ & $n\equiv 0\mod 4$ & $n+1$ \\ \hline
            D & $q\cong \gamma_{a,n}$ & $n\equiv 2\mod 4$ & $n+2$
        \end{tabular}
    \end{align*}
\end{theorem}
\begin{proof}
    Similar to the case of characteristic $\neq2$, it is easy to see that any $n$-dimensional form over a finite field contains $\gamma_{a,n-2}$ as a subform.
    By Lemma~\ref{lem:q_n-1AsSubformChar2}, we further know that when $n\equiv 0 \mod 4$,  $\gamma_{a,n-1}$ is dominated by $q$.
    The claim thus follows using \ref{LemUpperBoundClique}.
\end{proof}

\begin{remark}\label{rem:TwoOrbits}
    In case B in Theorem~\ref{CliqueNumbersTheoremChar2}, we have the unique situation that the dimension of the subspace spanned by a maximum clique is not uniquely determined. The dimension can be either $n-1$ or $n-2$. Therefore we have maximum cliques in $\red V$ of two different types.
    
    The cliques consisting of $n-1$ linearly independent vectors come from domination of the form $\gamma_{a,n-1}$, and cliques consisting of $n-1$ linearly dependent vectors come from the subform $\gamma_{a,n-2}$ (see Lemma~\ref{LemUpperBoundClique}).
\end{remark}

\begin{corollary} \label{cor:co-dimChar2}
    In the situation of the Theorem~\ref{CliqueNumbersTheoremChar2}, $q$ contains a subform isometric to $\gamma_{a,k}$ with
    \[
        k = 
        \begin{cases}
            n-2 & \text{in case A}\\
            n-1 & \text{in case B}\\
            n & \text{in cases C and D}.
        \end{cases}
    \]
\end{corollary}

\begin{remark}
    Lemma~\ref{BilinearFormCliqueReducedGraph} shows that a clique in $V(\red G)$ is an $(a, a^2)$-equiangular system, see \cite[Definition 4.1]{MR4339582}. 
    Over perfect fields of characteristic $2$ such as finite fields, these notations coincide since every element has exactly one square root.
    In general, equiangular systems may be bigger \cite[Theorem 4.2]{MR4339582} and thus, associated graphs may have a higher clique number than the ones we just calculated in our graphs, see \cite[Lemma 6.2]{MR4364998}.
\end{remark}

\subsection{Number of maximum cliques}

We now turn our investigation to the calculation of the number of maximal cliques for the case of a finite field. 
We will denote this number by $\omega_\text{max}(\Gcal_{q,a})$. 

\begin{proposition}\label{AllMaxCliquesInOneOrbit}
    Let $C$ be a clique of maximal cardinality $k$ in $\Gcal_{q,a}$ for some quadratic form over a field and consider the set $\binom Vk$ of all subsets of $k$ elements of $V$. We let $\iso(q)$ act on $\binom Vk$ by 
    $$\sigma(\{x_1,\ldots,x_k\}):=\{\sigma(x_1),\ldots, \sigma(x_k)\}.$$
    Then the orbit of $C$ consists precisely of all maximum cliques except for case B in Theorem~\ref{CliqueNumbersTheoremChar2}, where the set of maximum cliques is a union of two orbits.
\end{proposition}
\begin{proof}
    It is clear that the orbit of $C$ only contains cliques of maximal cardinality.
    Let now $C_1, C_2$ be two cliques of maximal cardinality. 
    After translation, we may assume $0\in C_1, C_2$. 
    For $i\in\{1,2\}$, let $C_i=\{0, x^{(i)}_1,\ldots, x^{(i)}_{k-1}\}$.
    By Lemma~\ref{LemUpperBoundClique} and Remark~\ref{rem:clique:independent} either $x^{(i)}_1,\ldots, x^{(i)}_{k-1}$ are linearly independent or, after renumbering, we have $x^{(i)}_{k-1}=-\sum\limits_{j=1}^{k-2}x^{(i)}_j$ and $x^{(i)}_1,\ldots, x^{(i)}_{k-2}$ are linearly independent.
    
    We first show that if $\dim(\spn(C_1)) = \dim(\spn(C_2))$ they lie in the same orbit. In all cases except from case $B$ of Theorem~\ref{CliqueNumbersTheoremChar2} we will then obtain a unique orbit and by Remark~\ref{rem:TwoOrbits} we will obtain two distinct orbits in the exceptional case.
    
    If $\dim(\spn(C_1)) = k-1 = \dim(\spn(C_2))$, let $\sigma\in S_{k-1}$ be any permutation of $\{1,\ldots, k-1\}$.
    There is a unique linear map $\spn(C_1)\to \spn(C_2)$ with $x^{(1)}_{j}\mapsto x^{(2)}_{\sigma(j)}$.
    This is clearly a linear isometry that can be extended to an isometry of $q$ by Witt's extension theorem \cite[Theorem 8.3]{MR2427530}.\\
    In the remaining case we use the same strategy as above.
    Here we can define $\sigma$ by choosing elements $j_1,\ldots, j_{k-2}$ with $\{j_1,\ldots, j_{k-2}\}\in\binom{\{1,\ldots, k-1\}}{k-2}$, map $x^{(1)}_\ell\mapsto x^{(2)}_{j_{\ell}}$ and extend this to an isometry.
\end{proof}

The above argument in particular shows the following:

\begin{corollary}\label{LinearIsometriesBetweenCliquesInRedV}
     Let $C_1, C_2\subseteq \red V$ be cliques of size $k$ with $0\in C_1, C_2$ and $\dim(\spn(C_1))=\dim(\spn(C_2))$. 
     Then there are $(k-1)!$  linear isometries $$(\spn(C_1), q)\to(\spn(C_2), q)$$ that map $C_1$ to $C_2$.
\end{corollary}

\begin{proposition}\label{SizeStabilizator}
    Let $C\subseteq \red V$ be a maximum clique of size $k$ with $0\in C$ and let $W=\spn(C)$. 
    For the stabilizer $\iso(q|_W)_C$, we have 
    $$|\iso(q|_W)_C|=k!.$$
\end{proposition}
\begin{proof}
    Let $\varphi\in \iso(q)_C$. 
    There are unique $v\in V, \sigma\in O(q)$ with $\varphi=\tau_v\circ\sigma$. 
    Let $C=\{0, x_1\ldots, x_{k-1}\}$. We then have
    \begin{align*}
        \{\sigma(0), \sigma(x_1),\ldots, \sigma({x_{k-1}})\}=\{-v, x_1-v,\ldots, x_{k-1}-v\}.
    \end{align*}
    By Lemma~\ref{TranslationOfCliquesInVred}, this equality is possible if and only if $v\in C$, i.e. we have $k$ possible choices for $v$. Combining with Corollary~\ref{LinearIsometriesBetweenCliquesInRedV}, we obtain the desired result. 
\end{proof}

As a last auxiliary result, we need the size of the isometry groups $\iso(q)$.

\begin{theorem}[\cite{MR1859189}, Theorem 9.11 and Theorem 14.48]\label{SizeOfIso(q)}
    Let $q$ be a quadratic form of dimension $n$ over a finite field with $f$ elements. We have 
    \begin{align*}
        |O(q)|=\begin{cases}
            2\cdot f^{\frac{n(n-1)}{2}}\cdot\prod\limits_{2i<n}(1-f^{-2i}),&\text{ if }n\text{ is odd},\\
            2\cdot f^{\frac{n(n-1)}{2}}\cdot(1-f^{-\frac{n}{2}})\cdot\prod\limits_{2i<n}(1-f^{-2i}),&\text{ if }q\text{ is hyperbolic},\\
            2\cdot f^{\frac{n(n-1)}{2}}\cdot(1+f^{-\frac{n}{2}})\cdot\prod\limits_{2i<n}(1-f^{-2i}),&\text{ if }n\text{ is even, }q\text{ not hyperbolic}
        \end{cases}
    \end{align*}
    In particular, we have 
    \begin{align*}
        |O(q)|=\begin{cases}
            2,&\text{ if }\dim(q)=1\\
            2f-2,&\text{ if }q\cong\mathbb H\\
            2f+2,&\text{ if }\dim(q)=2, q\not\cong\mathbb H\\
        \end{cases}
    \end{align*}
\end{theorem}

Since we clearly have $|\iso(q)|=f^n\cdot |O(q)|$, we obtain:

\begin{corollary}\label{cor:SizeIso}
    Let $q$ be a quadratic form of dimension $n$ over a finite field with $f$ elements.
    We have the following values for $|\iso(q)|$:
    \begin{itemize}
        \item 
    If $q$ is hyperbolic, we have 
    \begin{align*}
        |\iso(q)|=2\cdot f^{\frac{n(n+1)}{2}}\cdot(1-f^{-\frac{n}{2}})\cdot\prod\limits_{2i<n}(1-f^{-2i})
    \end{align*}
    \item If $\dim(q)$ is even, but $q$ is not hyperbolic, we have 
    \begin{align*}
        |\iso(q)|=2\cdot f^{\frac{n(n+1)}{2}}\cdot(1+f^{-\frac{n}{2}})\cdot\prod\limits_{2i<n}(1-f^{-2i})
    \end{align*}
    \item If $\dim(q)$ is odd, we have 
    \begin{align*}
        |\iso(q)| = 2\cdot f^{\frac{n(n+1)}{2}}\cdot\prod\limits_{2i<n}(1-f^{-2i})
    \end{align*}
    \end{itemize}
\end{corollary}

{
\renewcommand{\arraystretch}{1.5}
\begin{theorem}\label{NumberOfCliques}
    Let $F$ be a finite field with $f$ elements, $q$ be an $n$-dimensional quadratic form over $F$ and $a\in F^*$. 
    We use the notation of Corollary~\ref{cor:SizeIso}.
    Let $\alpha$ be defined by
    \begin{align*}
        \alpha=\begin{cases}
            2f-2,&\text{ if }q\perp \gamma_{a,n-2}\text{ is hyperbolic},\\
            2f+2,&\text{ otherwise.}
        \end{cases}
    \end{align*}
    For the number of maximum cliques $\omega_\text{max}(\Gcal_{q,a})$, we have the following values:
    \begin{enumerate}[label=(\alph*)]
        \item if $\cha F\neq 2$ (cases as in Theorem~\ref{CliqueNumbersTheorem}):
    \begin{align*}
        \begin{tabular}{c||c|c|c}
            case / isometry & $q$ hyperbolic & $dim(q)$ even, $q$ not hyperbolic & $\dim(q)$ odd\\ \hline\hline 
            A & 
            $\frac{|\iso(q)|}{\alpha\cdot n!}$ & 
            $\frac{|\iso(q)|}{\alpha\cdot n!}$ & 
            $\frac{|\iso(q)|}{\alpha \cdot n!}$\\ \hline
            B & 
            $\frac{|\iso(q)|}{2\cdot n!}$ & 
            $\frac{|\iso(q)|}{2 \cdot n!}$ & 
            $\frac{|\iso(q)|}{2 \cdot n!}$\\ \hline
            C & 
            $\frac{|\iso(q)|}{2(n+1)!}$ & 
            $\frac{|\iso(q)|}{2(n+1)!}$ & 
            $\frac{|\iso(q)|}{2(n+1)!}$\\ \hline
            D & 
            $\frac{|\iso(q)|}{(n+1)!}$ & 
            $\frac{|\iso(q)|}{(n+1)!}$ & 
            $\frac{|\iso(q)|}{(n+1)!}$\\\hline 
            E & 
            $\frac{|\iso(q)|}{(n+2)!}$ & 
            $\frac{|\iso(q)|}{(n+2)!}$ & 
            $\frac{|\iso(q)|}{(n+2)!}$
        \end{tabular}
    \end{align*}
    \item if $\cha F=2$ (cases as in Theorem~\ref{CliqueNumbersTheoremChar2}):
        \begin{align*}
            \begin{tabular}{c||c|c}
                case / isometry & $q$ hyperbolic & $q$ not hyperbolic\\\hline\hline
                A & $\frac{|\iso(q)|}{\alpha(n-1)!}$ & $\frac{|\iso(q)|}{\alpha(n-1)!}$ \\\hline
                B & $\frac{|\iso(q)|}{\alpha\cdot n!} + \frac{|\iso(q)|}{2\cdot n!}$ & $\frac{|\iso(q)|}{\alpha\cdot n!} + \frac{|\iso(q)|}{2\cdot n!}$\\\hline 
                C & $\frac{|\iso(q)|}{(n+1)!}$ & 
            $\frac{|\iso(q)|}{(n+1)!}$ \\\hline
            D & $\frac{|\iso(q)|}{(n+2)!}$ & 
            $\frac{|\iso(q)|}{(n+2)!}$ 
            \end{tabular}
        \end{align*}
    \end{enumerate}
\end{theorem}
}
\begin{proof}
    \begin{enumerate}[label=(\alph*)]
        \item 
    By Proposition~\ref{AllMaxCliquesInOneOrbit} we need to find the length of the orbit of a maximum clique under the action of $\iso(q)$.
    To do so, by basic group theory, we have to evaluate the term 
    $$\frac{|\iso(q)|}{|\iso(q)_C|}.$$
    By Corollary~\ref{cor:SizeIso}, we know the size of $\iso(q)$.
    
    We already know the size of the stabilizer of a maximum clique (containing $0$ lying in $\red V\cup\{0\}$) restricted to its span from \ref{SizeStabilizator}. Our task is thus reduced to the problem of finding on how many ways affine isometries of $q|_{\spn(C)}$ that fix $C$ can be extended to affine isometries of $q$.
    
    Since the translation part is already fixed (cf. Lemma~\ref{TranslationOfCliquesInVred}), it is enough to consider the case of linear isometries.
    Recall that $\spn(C)$ has an orthogonal complement of dimension $\leq 2$ in $V$ (cf. Corollary~\ref{cor:co-dim}). Since orthogonality is preserved under isometries, we just need to multiply $|\iso(q|_W)_C|=k!$ obtained in \ref{SizeStabilizator} with the size of the orthogonal group of the complement, which can be found in Theorem~\ref{SizeOfIso(q)}. The claim thus follows.
        \item The arguments for fields of characteristic 2 remain the same as in characteristic $\neq 2$ except for case B in Theorem~\ref{CliqueNumbersTheoremChar2} as pointed out in Proposition~\ref{AllMaxCliquesInOneOrbit}. 
        We thus now consider the case in which we have $q\not\cong \gamma_{a,n}$ but $q$ dominates $\gamma_{a,n-1}$. 
        We thus have a clique $C=\{0, x_1,\ldots, x_{n-1}\}$ with $x_1,\ldots, x_{n-1}$ linear independent and $q(x_i)=a$ for all $i\in\{1,\ldots, n-1\}$. 
        
        Let $V$ be the vector space on which $q$ is defined.
        We have to determine the number of isometries $\sigma:V\to V$ such that $\sigma(C)=C$.
        By Witt's extension theorem \cite[Theorem 8.3]{MR2427530}, we know there is at least one such isometry $\varphi$ that we fix from now on.
        Let $U\subseteq V$ be the subspace generated by $x_1,\ldots, x_{n-2}$.
        Since $q|_U\cong \gamma_{a,n-2}$ is non-degenerate, it has an orthogonal complement in $\spn(x_1,\ldots, x_{n-1})$ of dimension 1 with basis $v$. 
        We thus have an isometry $q\cong \gamma_{a,n-2}\perp [a, b]$ for some $b\in F$, where we may assume $v$ to be the vector corresponding to $a$, see Lemma~\ref{lem:q_n-1AsSubformChar2}. 
        Let $w$ be the vector corresponding to $b$, i.e. we have $w\in U^\perp$, $b(v, w)=1, q(w)=b$. 
        Since we have $\sigma(C)=C$, we have $\sigma(\spn(C))=\spn(C)$ and we are reduced to the task of determining the number of possible images of $w$.
        
        Since $\sigma(w)\in \sigma(U^\perp)=\sigma(U)^\perp=\varphi(U)^\perp$, we have 
        $$\sigma(w)\in\spn(\varphi(v), \varphi(w)),$$ i.e. there are $\lambda, \mu\in F$ with $\sigma(w)=\lambda\varphi(v)+\mu\varphi(v)$.
        To abbreviate the notation, we write $x:=\varphi(v), y:=\varphi(w)$.
        We have
        \begin{align*}
            1=b(v, w)&=b(\sigma(v), \sigma(w)) \\
            &= b(x, \lambda x+\mu y)=\lambda b(x,x)+\mu b(x,y)=\mu.
        \end{align*}
        We further have 
        \begin{align*}
            b=q(w)&=q(\sigma(w))\\
            &=q(\lambda x+y)\\
            &=\lambda^2q(x)+q(y)+\lambda b(x,y)\\
            &=\lambda^2\cdot a+b+\lambda\\
            \iff &\lambda\cdot(\lambda\cdot a+1)=0,
        \end{align*}
        i.e. $\lambda\in\{0, a^{-1}\}$. 
        A routine check shows that both choices lead to an isometry, where we obtain $\varphi$ when $\lambda=0$.
        Combining with Proposition~\ref{SizeStabilizator}, we obtain that the size of the stabilizer of such cliques is given by $2n!$.
        Plugging in the values gives the assertion.
    \end{enumerate}
\end{proof}

\begin{remark}
    We would like to recall that we only have to check dimension and determinant resp. Arf-invariant in order to identify the isometry type of a quadratic form over a finite field. Thus theorems~\ref{SizeOfIso(q)} and \ref{NumberOfCliques} together completely solve the task of finding the number of maximum cliques.
\end{remark}

\begin{example}
    We will show how we can explicitly calculate a clique for a given quadratic form. We consider the field $F=\mathbb F_5$, $a=1$ and the quadratic form $q=\qf{1, 1, 2}$ for the canonical basis $\mathcal B = (e_1, e_2, e_3)$. The Gram matrix for $b = b_q$ with respect to $\mathcal B$ is the given by \begin{align*}
        G:=\mathcal G_{b, \mathcal B}=\begin{pmatrix}
            2 & 0 & 0\\ 0 & 2 & 0\\ 0 & 0 & 4
        \end{pmatrix}
    \end{align*}
    and has determinant 1. Further $A:=I + J$ has determinant 1 as well, so that we already know we will find a clique of size 5 in $\Gcal_{q, 1}$. Using the well known algorithm from linear algebra to diagonalize Gram matrices, we readily see that for 
    \begin{align*}
        S:=\begin{pmatrix}
            1 & 1 & 2\\ 0 & 1 & 1\\ 0 & 2 & 0
        \end{pmatrix}
    \end{align*}
    we have $S^TAS = G$ or equivalently, we have $A = (S^{-1})^T G S^{-1}$ where $S^{-1}$ is given by
    \begin{align*}
        \begin{pmatrix}
            1 & 3 & 3\\ 0 & 0 & 3\\ 0 & 1 & 2
        \end{pmatrix}.
    \end{align*}
    By the above results and its proofs, a clique is thus given by
    \begin{align*}
        \begin{pmatrix}
            0\\0\\0
        \end{pmatrix},
        \begin{pmatrix}
            1\\0\\0
        \end{pmatrix}, 
        \begin{pmatrix}
            3\\0\\1
        \end{pmatrix}, 
        \begin{pmatrix}
            3\\3\\2
        \end{pmatrix}, 
        -\left(\begin{pmatrix}
            1\\0\\0
        \end{pmatrix} 
        +
        \begin{pmatrix}
            3\\0\\1
        \end{pmatrix}
        +
        \begin{pmatrix}
            3\\3\\2
        \end{pmatrix}
        \right) = 
        \begin{pmatrix}
            3\\2\\2
        \end{pmatrix}.
    \end{align*}
    In total, we have
    \begin{align*}
        \frac{2\cdot 5^{\frac{3\cdot 4}{2}}(1-5^{-2})}{4!}=2\cdot5^4=1250
    \end{align*}
    maximum cliques.
\end{example}

\section{Fields of characteristic $0$} \label{sec:characteristic:zero}

\subsection{Sum of squares form over fields of characteristic $0$}

In \cite{Krebs_2022} M. Krebs used techniques developed before in \cite{Chilakamarri_1988} in order to solve the problem of determining the clique number of the quadratic form $s_n:=n\times \qf{1}$ over the rational numbers $\Q$ and adapted them to find cliques over the graphs for finite fields. 

Our algebraic approach is also able to easily recover the results of the latter paper as we will now show. Since we only consider the case $a=1$, we write shortly $\gamma_n:=\gamma_{1, n}$. We will start with a technical lemma. 

\begin{lemma}\label{q_anOverQ}
    Let $F$ be a field of characteristic 0 and $n\in\N$ be an integer. We have an isometry 
    $$\gamma_{n}\cong\qf{2\cdot1\cdot2,2\cdot2\cdot3,\ldots, 2\cdot n\cdot(n+1)}.$$
\end{lemma}
\begin{proof}
    We use induction on $n$. 
    For $n=1$, the Gram matrix of the associated bilinear form is given by $(2)$ and thus, $\gamma_{1}$ is isometric to $\qf{2\cdot2}$ as desired. 
    For $n\geq 2$ we consider the invertible matrix
    \begin{align*}
        S:=\begin{pmatrix}
            1 & 1 & \ldots & \ldots & 1\\
            1 & -1 & 0 & \ldots & 0\\
            \vdots & 0 & \ddots & \ddots & \vdots\\
            \vdots & \vdots & \ddots & \ddots & 0\\
            1 & 0 & \ldots & 0 & -1
        \end{pmatrix}.
    \end{align*}
    We now obtain the conclusion by calculating
    \begin{align*}
        S^T(I_n+J_n)S=\left(\begin{array}{c|ccc}
             n\cdot(n+1) & 0 &\ldots & 0 \\
             \hline
             0 & & & \\
             \vdots & & I_{n-1}+J_{n-1} & \\
             0 & & & 
        \end{array}\right)
    \end{align*}
    and applying the induction hypothesis.
\end{proof}

\begin{lemma}
    Let $F$ be a field of characteristic 0 and let $n\in\N$ be an integer. We have an isometry $\qf{2(n-1)n, 2n(n+1)}\cong\qf{1, n^2-1}$.
    In particular we have $\gamma_{n-1}\cong \gamma_{n-3}\perp\qf{1, n(n-2)}$.
\end{lemma}
\begin{proof}
    Using $\qf{a,b}\cong\qf{a+b,ab(a+b)}$ if $a+b\neq0$ and manipulating modulo squares, we have
    \begin{align*}
        \qf{2(n-1)n, 2n(n+1)}&\cong\qf{4n^2, 4n^2\cdot 4n^2(n^2-1)}\\
        &\cong\qf{1, n^2-1}.
    \end{align*}
    The result for $\gamma_{n-1}$ now follows from Lemma~\ref{q_anOverQ}, the above calculation and 
    $${(n-1)^2-1=n^2-2n=n(n-1)}.$$
\end{proof}

\begin{proposition}\label{Char0Qn-1InSn}
    Let $F$ be a field of characteristic 0. For even $n\in\N$, we have $s_n\cong \gamma_{n-1}\perp \qf{2n}$.
    In particular we have $s_n\cong \gamma_n$ if and only if $n+1$ is a square.
\end{proposition}
\begin{proof}
    For $n=2$, we have 
    \begin{align*}
        s_2=\qf{1,1}\cong\qf{4,4}=\gamma_1\perp\qf{2\cdot 2}.
    \end{align*}
    For $n\geq4$ we have by induction
    \begin{align*}
        s_n&=s_{n-2}\perp\qf{1,1}\\
        &=\gamma_{n-3}\perp\qf{\det(\gamma_{n-3})}\perp\qf{1,1}\\
        &=\gamma_{n-3}\perp \qf{1,1,2^{n-3}(n-2)}\\
        &=\gamma_{n-3}\perp\qf{1,1,2(n-2)}\\
        &=\gamma_{n-3}\perp \qf{1,(n-2)^2, 2(n-2)}\\
        &=\gamma_{n-3}\perp \qf{1, n(n-2), n(n-2)(n-2)^2\cdot2(n-2)}\\
        &=\gamma_{n-3}\perp \qf{1, n(n-2)}\perp\qf{2n}=\gamma_{n-1}\perp\qf{2n}.
    \end{align*}
    Using Witt cancellation we thus have
    \begin{align*}
        \gamma_{n-1}\perp\qf{2n}\cong s_n\cong \gamma_n=\gamma_{n-1}\perp \qf{2n(n+1)}&\iff \qf{2n}\cong\qf{2n(n+1)},
    \end{align*}
    i.e. if and only if $n+1$ is a square.
\end{proof}

The following theorem is now clear.

\begin{theorem}{\cite[Theorem 7]{Chilakamarri_1988}}
    Let $F$ be a field of characteristic 0. For n even, we have 
    \begin{align*}
        \omega(\Gcal_{1, s_n})=\begin{cases}
            n,& \text{if }n\text{ is not a square}\\
            n+1, & \text{otherwise}
        \end{cases}.
    \end{align*}
\end{theorem}

We now turn our attention to odd dimensional forms. 
By using the even dimensional case from Proposition~\ref{Char0Qn-1InSn}, we easily see that for odd $n$, we have
\begin{align*}
    s_n&=s_{n-1}\perp\qf1\\
    &=\gamma_{n-2}\perp\qf{2(n-1)}\perp\qf1
\end{align*}

We thus have $\gamma_{n-1}\subseteq s_n$ if and only if $\qf{1, 2(n-1)}$ represents $2(n-1)n$. 
Using the multiplicativity of Pfister forms, this is the case if and only if $\qf{1, 2(n-1)}$ represents $n$. 
Finally we have $s_n\cong \gamma_n$ if and only if
\begin{align*}
    \qf{1,2(n-1)}\cong\qf{2(n-1)n, 2n(n+1)},
\end{align*}
i.e. if and only if $\qf{1, 2(n-1)}$ represents $n$ and
\begin{align*}
    2(n-1)=\det(\qf{1,2(n-1)})=\det(\qf{2(n-1)n, 2n(n+1)})=(n-1)(n+1)
\end{align*}
where the determinants are interpreted modulo squares as usual.

We thus have proven the following:

\begin{theorem}{\cite[Theorem 9]{Chilakamarri_1988}}
    Let $F$ be a field of characteristic 0. For n odd, we have 
    \begin{align*}
        \omega(\Gcal_{1, s_n})=\begin{cases}
            n-1,& \text{if }n\notin D_F(\qf{1,2(n-1)})\\
            n,& \text{if }n\in D_F(\qf{1,2(n-1)})\text{ and }2(n+1)\notin F^{\ast2}\\
            n+1,& \text{if }n\in D_F(\qf{1,2(n-1)})\text{ and }2(n+1)\in F^{\ast2}
        \end{cases}.
    \end{align*}
\end{theorem}

\subsection{Real numbers}

Recall that a non-degenerate quadratic form $q$ over $\R$ is uniquely determined by its signature $(r_+,r_-)$ where $r_+$ is the number of positive entries in a diagonalization and $r_-$ the number of negative entries in a diagonalization. 
We further see by \ref{MaxCliqueForm} that $\gamma_{a,n}$ is positive resp. negative definite if and only if $a$ is positive resp. negative and thus isomorphic to $n\times\qf{1}$ resp. $n\times\qf{-1}$.

We thus have the following:

\begin{theorem}
    Let $a\in\R$ with $a\neq0$ and $q$ be a quadratic form over the reals of signature $(r_+, r_-)$. 
    We then have 
    \begin{align*}
        \omega(\Gcal_{q, a})\begin{cases}
            r_++1&,\text{ if }a>0\\
            r_-+1&,\text{ if }a<0
            \end{cases}.
    \end{align*}
\end{theorem}

\subsection{$p$-adic numbers and rational numbers}

Every quadratic form $\varphi$ over $\Q_p$ with odd $p$ can be written as 
$$\varphi\cong\qf{a_1,\ldots, a_k}\perp p\qf{b_1,\ldots, b_{\ell}}$$ 
with integers $a_1,\ldots, a_k, b_1,\ldots, b_\ell\in\Z$ not divisible by $p$. 
If we fix such a decomposition, we can then define quadratic forms
\begin{align*}
    \varphi_1:=\qf{\overline{a_1},\ldots, \overline{a_k}},~~~
    \varphi_2:=\qf{\overline{b_1},\ldots, \overline{b_\ell}}
\end{align*}
over $\mathbb F_p$, where the bar denotes reducing modulo $p$. 
We then have 
$$i_W(\varphi)=i_W(\varphi_1)+i_W(\varphi_2).$$
We know by the results of the previous section that determining the size of a maximal clique is equivalent to finding certain subforms. Now using Lemma~\ref{SubformWittIndex} and Theorem~\ref{CliqueNumbersTheorem}, we are thus able to determine the clique number of a graph given by an arbitrary quadratic form over $\Q_p$.

Even though the quadratic form theory over $\Q_2$ is more involved, it is well understood. 
For example, we know from \cite[Chapter VI. Corollary 2.24]{MR2104929} that a basis of the 8 square classes is given by $\{-1, 2, 5\}$.
We have 32 non-equivalent Witt classes in $\Q_2$ which are all generated by the forms $\qf{1}, \qf{1, -2}, \qf{1, -5}$, see \cite[Chapter VI. Corollary 2.30, Remark 2.31]{MR2104929}, where the explicit ring structure of $W\Q_2$ is calculated.
For a given form it is thus possible to solve the problem over this field using basic quadratic form calculations. 
Nevertheless it is not convenient to fully describe the multitude of all the different cases.

We now turn to determining the clique number of graphs induced by an arbitrary quadratic form $q$ over the rationals $\Q$. 
By the above results, we need to find the highest integer $n$ such that the test form $\gamma_{a,n}$ as defined before is a subform of $q$. 
By Lemma~\ref{SubformWittIndex} this is the highest integer $n$ such that $q\perp- \gamma_{a,n}$ has Witt index at least $\dim(\gamma_{a,n})=n$. 
By a corollary of the Hasse-Minkowski Theorem \cite[19.4]{MR2788987} the Witt index of $q\perp- \gamma_{a,n}$ is given by the minimum of the Witt indices $q\perp-\gamma_{a,n}$ considered as a form over the reals and the $p$-adic numbers for all primes $p$.
We have therefore reduced the problem to cases we have already solved and can give the following local-global-principle for cliques in graphs attached to quadratic forms over $\Q$.

\begin{theorem}[Local-global principle for clique numbers over $\Q$] \label{thm:local:global}
    Let $q$ be a quadratic form over $\Q$ of dimension $n$ and $a\in \Q^*$. 
    We have
    \begin{align*}
        d:=&\max\{k\in\N_0\mid \gamma_{k, a}\subseteq q\}\\
        =&\max\{k\in\N_0\mid \min\{i_W((q\perp -\gamma_{a,k})_p)\mid p\text{ is a prime or }\infty\}\geq k\}.
    \end{align*}
    In particular, we have $\omega(\Gcal_{q, a})=d+1$.
\end{theorem}

In the final remark, we collect some useful calculation rules.

\begin{remark}
    Let $q$ be a quadratic form over the rationals of dimension $n$ with signature $(r_+, r_-)$ and let $a\in \Q^*$. Let further $k\in\N_0$ be maximal with $\gamma_{a,k}\subseteq q$.
    \begin{enumerate}[label=(\alph*)]
        \item For $a>0$, we have $k\leq r_+$ since we have $\gamma_{a, k}\subseteq q$ over the reals. 
        But since $\gamma_{a,k}$ is positive definite, it follows $k\leq r_+$.
        Similarly, for $a<0$, we have $k\leq r_-$.
        \item For every prime $p$, we have $u(\Q_p)=4$, i.e. 4 is the maximum dimension of an anisotropic quadratic form over $\Q_p$. 
        In particular,for all $\ell\in\N_0$, we have
        $$i_W(q\perp-\gamma_{a,\ell})_p\geq \lfloor\frac{n+\ell-3}{2}\rfloor.$$
        We have $\lfloor\frac{n+\ell-1}{2}\rfloor\geq\ell$ for all $\ell\leq n-3$.
        \item For fixed diagonalizations of $q$ and $\gamma_{a,\ell}$, let $\Pi$ be the set of primes $\neq 2$ not occuring in any of these diagonalizations. 
        If $p\in\Pi$, $q\perp- \gamma_{a,\ell}$ is unimodular and has thus Witt index at least 
        $$\lfloor\frac{n+\ell-1}{2}\rfloor.$$
        We have $\lfloor\frac{n+\ell-1}{2}\rfloor\geq\ell$ for all $\ell\leq n-1$.
    \end{enumerate}
\end{remark}

\begin{example}
    Since it does not seem to be convenient to write down all case distinctions separately, we will give an example that shows how our theory can be applied. 
    Let $q=\qf{1, 2, 3, -7}$ and $a=1$. 
    By considering the real numbers and since we already know that $\gamma_{1,n}$ is positive definite for all $n\in\N$, we know that we cannot have $\gamma_{1,4}\subseteq q$.
    Further, we have $\gamma_{1,2}\cong\qf{1,3}$ and $\gamma_{1,3}\cong\qf{1,3,6}$.
    We obviously have $\gamma_{1,2}\subseteq q$ and thus find a clique of size 3.
    In order to decide whether we have $\gamma_{1,3}\subseteq q$, we only have to decide if $\qf{2, -6, -7}$ is isotropic.
    Considering this form over the 3-adic numbers $\Q_3$ with decomposition $\qf{2, -6, -7}\cong \qf{1,1}\perp 3\qf{1}$ into residue class forms, we see that this form is anisotropic over $\Q_3$ and thus also over $\Q$.
    We thus have $\omega(\Gcal_{q,1})=3$.
\end{example}

%%===========================================================================================%%
%% If you are submitting to one of the Nature Portfolio journals, using the eJP submission   %%
%% system, please include the references within the manuscript file itself. You may do this  %%
%% by copying the reference list from your .bbl file, paste it into the main manuscript .tex %%
%% file, and delete the associated \verb+\bibliography+ commands.                            %%
%%===========================================================================================%%

\bibliography{sn-bibliography}% common bib file

%% BioMed_Central_Bib_Style_v1.01

\begin{thebibliography}{16}
% BibTex style file: bmc-mathphys.bst (version 2.1), 2014-07-24
\ifx \bisbn   \undefined \def \bisbn  #1{ISBN #1}\fi
\ifx \binits  \undefined \def \binits#1{#1}\fi
\ifx \bauthor  \undefined \def \bauthor#1{#1}\fi
\ifx \batitle  \undefined \def \batitle#1{#1}\fi
\ifx \bjtitle  \undefined \def \bjtitle#1{#1}\fi
\ifx \bvolume  \undefined \def \bvolume#1{\textbf{#1}}\fi
\ifx \byear  \undefined \def \byear#1{#1}\fi
\ifx \bissue  \undefined \def \bissue#1{#1}\fi
\ifx \bfpage  \undefined \def \bfpage#1{#1}\fi
\ifx \blpage  \undefined \def \blpage #1{#1}\fi
\ifx \burl  \undefined \def \burl#1{\textsf{#1}}\fi
\ifx \doiurl  \undefined \def \doiurl#1{\url{https://doi.org/#1}}\fi
\ifx \betal  \undefined \def \betal{\textit{et al.}}\fi
\ifx \binstitute  \undefined \def \binstitute#1{#1}\fi
\ifx \binstitutionaled  \undefined \def \binstitutionaled#1{#1}\fi
\ifx \bctitle  \undefined \def \bctitle#1{#1}\fi
\ifx \beditor  \undefined \def \beditor#1{#1}\fi
\ifx \bpublisher  \undefined \def \bpublisher#1{#1}\fi
\ifx \bbtitle  \undefined \def \bbtitle#1{#1}\fi
\ifx \bedition  \undefined \def \bedition#1{#1}\fi
\ifx \bseriesno  \undefined \def \bseriesno#1{#1}\fi
\ifx \blocation  \undefined \def \blocation#1{#1}\fi
\ifx \bsertitle  \undefined \def \bsertitle#1{#1}\fi
\ifx \bsnm \undefined \def \bsnm#1{#1}\fi
\ifx \bsuffix \undefined \def \bsuffix#1{#1}\fi
\ifx \bparticle \undefined \def \bparticle#1{#1}\fi
\ifx \barticle \undefined \def \barticle#1{#1}\fi
\bibcommenthead
\ifx \bconfdate \undefined \def \bconfdate #1{#1}\fi
\ifx \botherref \undefined \def \botherref #1{#1}\fi
\ifx \url \undefined \def \url#1{\textsf{#1}}\fi
\ifx \bchapter \undefined \def \bchapter#1{#1}\fi
\ifx \bbook \undefined \def \bbook#1{#1}\fi
\ifx \bcomment \undefined \def \bcomment#1{#1}\fi
\ifx \oauthor \undefined \def \oauthor#1{#1}\fi
\ifx \citeauthoryear \undefined \def \citeauthoryear#1{#1}\fi
\ifx \endbibitem  \undefined \def \endbibitem {}\fi
\ifx \bconflocation  \undefined \def \bconflocation#1{#1}\fi
\ifx \arxivurl  \undefined \def \arxivurl#1{\textsf{#1}}\fi
\csname PreBibitemsHook\endcsname

%%% 1
\bibitem[\protect\citeauthoryear{Medrano
  et~al.}{1996}]{medrano_myers_stark_terras_1996}
\begin{barticle}
\bauthor{\bsnm{Medrano}, \binits{A.}},
\bauthor{\bsnm{Myers}, \binits{P.}},
\bauthor{\bsnm{Stark}, \binits{H.M.}},
\bauthor{\bsnm{Terras}, \binits{A.}}:
\batitle{Finite analogues of euclidean space}.
\bjtitle{Journal of Computational and Applied Mathematics}
\bvolume{68}(\bissue{1-2}),
\bfpage{221}--\blpage{238}
(\byear{1996})
\doiurl{10.1016/0377-0427(95)00261-8}
\end{barticle}
\endbibitem

%%% 2
\bibitem[\protect\citeauthoryear{Bannai
  et~al.}{2009}]{bannai_shimabukuro_tanaka_2009}
\begin{barticle}
\bauthor{\bsnm{Bannai}, \binits{E.}},
\bauthor{\bsnm{Shimabukuro}, \binits{O.}},
\bauthor{\bsnm{Tanaka}, \binits{H.}}:
\batitle{Finite euclidean graphs and ramanujan graphs}.
\bjtitle{Discrete Mathematics}
\bvolume{309}(\bissue{20}),
\bfpage{6126}--\blpage{6134}
(\byear{2009})
\doiurl{10.1016/j.disc.2009.06.008}
\end{barticle}
\endbibitem

%%% 3
\bibitem[\protect\citeauthoryear{Medrano et~al.}{1998}]{medrano:1998:feg}
\begin{barticle}
\bauthor{\bsnm{Medrano}, \binits{A.}},
\bauthor{\bsnm{Myers}, \binits{P.}},
\bauthor{\bsnm{Stark}, \binits{H.}},
\bauthor{\bsnm{Terras}, \binits{A.}}:
\batitle{Finite euclidean graphs over rings}.
\bjtitle{Proceedings of the American Mathematical Society}
\bvolume{126}(\bissue{3}),
\bfpage{701}--\blpage{710}
(\byear{1998})
\end{barticle}
\endbibitem

%%% 4
\bibitem[\protect\citeauthoryear{Krebs}{2022}]{Krebs_2022}
\begin{barticle}
\bauthor{\bsnm{Krebs}, \binits{M.}}:
\batitle{Clique numbers of finite unit-quadrance graphs}.
\bjtitle{Journal of Algebraic Combinatorics}
(\byear{2022})
\doiurl{10.1007/s10801-022-01157-8}
\end{barticle}
\endbibitem

%%% 5
\bibitem[\protect\citeauthoryear{Chilakamarri}{1988}]{Chilakamarri_1988}
\begin{barticle}
\bauthor{\bsnm{Chilakamarri}, \binits{K.B.}}:
\batitle{Unit-distance graphs in rational $n$-spaces}.
\bjtitle{Discrete Mathematics}
\bvolume{69}(\bissue{3}),
\bfpage{213}--\blpage{218}
(\byear{1988})
\doiurl{10.1016/0012-365x(88)90049-0}
\end{barticle}
\endbibitem

%%% 6
\bibitem[\protect\citeauthoryear{Lam}{2005}]{MR2104929}
\begin{bbook}
\bauthor{\bsnm{Lam}, \binits{T.Y.}}:
\bbtitle{Introduction to Quadratic Forms over Fields}.
\bsertitle{Graduate Studies in Mathematics},
vol. \bseriesno{67},
p. \bfpage{550}.
\bpublisher{American Mathematical Society},
\blocation{Providence, RI}
(\byear{2005}).
\doiurl{10.1090/gsm/067} .
\burl{https://doi.org/10.1090/gsm/067}
\end{bbook}
\endbibitem

%%% 7
\bibitem[\protect\citeauthoryear{Baeza}{1978}]{MR0491773}
\begin{bbook}
\bauthor{\bsnm{Baeza}, \binits{R.}}:
\bbtitle{Quadratic Forms over Semilocal Rings}.
\bsertitle{Lecture Notes in Mathematics, Vol. 655},
p. \bfpage{199}.
\bpublisher{Springer},
\blocation{Providence, RI}
(\byear{1978})
\end{bbook}
\endbibitem

%%% 8
\bibitem[\protect\citeauthoryear{Ray-Chaudhuri}{1962}]{MR133725}
\begin{barticle}
\bauthor{\bsnm{Ray-Chaudhuri}, \binits{D.K.}}:
\batitle{Some results on quadrics in finite projective geometry based on
  {G}alois fields}.
\bjtitle{Canadian J. Math.}
\bvolume{14},
\bfpage{129}--\blpage{138}
(\byear{1962})
\doiurl{10.4153/CJM-1962-010-2}
\end{barticle}
\endbibitem

%%% 9
\bibitem[\protect\citeauthoryear{Zimmermann}{2017}]{zimmermann:2017:phd}
\begin{botherref}
\oauthor{\bsnm{Zimmermann}, \binits{M.C.}}:
Configurations of sublattices and dirichlet-voronoi cells of periodic point
  sets.
PhD thesis,
Dissertation, Technische Universit{\"a}t Dortmund
(2017)
\end{botherref}
\endbibitem

%%% 10
\bibitem[\protect\citeauthoryear{Ebeling}{2013}]{MR2977354}
\begin{bbook}
\bauthor{\bsnm{Ebeling}, \binits{W.}}:
\bbtitle{Lattices and Codes},
\bedition{3}rd edn.
\bsertitle{Advanced Lectures in Mathematics},
p. \bfpage{167}.
\bpublisher{Springer},
\blocation{Wiesbaden}
(\byear{2013}).
\doiurl{10.1007/978-3-658-00360-9} .
\bcomment{A course partially based on lectures by Friedrich Hirzebruch}.
\burl{https://doi.org/10.1007/978-3-658-00360-9}
\end{bbook}
\endbibitem

%%% 11
\bibitem[\protect\citeauthoryear{Kneser}{2002}]{MR2788987}
\begin{bbook}
\bauthor{\bsnm{Kneser}, \binits{M.}}:
\bbtitle{Quadratische {F}ormen},
p. \bfpage{164}.
\bpublisher{Springer}, \blocation{???}
(\byear{2002}).
\doiurl{10.1007/978-3-642-56380-5} .
\bcomment{Revised and edited in collaboration with Rudolf Scharlau}.
\burl{https://doi.org/10.1007/978-3-642-56380-5}
\end{bbook}
\endbibitem

%%% 12
\bibitem[\protect\citeauthoryear{Hoffmann and Laghribi}{2004}]{MR2058517}
\begin{barticle}
\bauthor{\bsnm{Hoffmann}, \binits{D.W.}},
\bauthor{\bsnm{Laghribi}, \binits{A.}}:
\batitle{Quadratic forms and {P}fister neighbors in characteristic 2}.
\bjtitle{Trans. Amer. Math. Soc.}
\bvolume{356}(\bissue{10}),
\bfpage{4019}--\blpage{4053}
(\byear{2004})
\doiurl{10.1090/S0002-9947-04-03461-0}
\end{barticle}
\endbibitem

%%% 13
\bibitem[\protect\citeauthoryear{Greaves et~al.}{2022a}]{MR4339582}
\begin{barticle}
\bauthor{\bsnm{Greaves}, \binits{G.R.W.}},
\bauthor{\bsnm{Iverson}, \binits{J.W.}},
\bauthor{\bsnm{Jasper}, \binits{J.}},
\bauthor{\bsnm{Mixon}, \binits{D.G.}}:
\batitle{Frames over finite fields: basic theory and equiangular lines in
  unitary geometry}.
\bjtitle{Finite Fields Appl.}
\bvolume{77},
\bfpage{101954}--\blpage{41}
(\byear{2022})
\doiurl{10.1016/j.ffa.2021.101954}
\end{barticle}
\endbibitem

%%% 14
\bibitem[\protect\citeauthoryear{Greaves et~al.}{2022b}]{MR4364998}
\begin{barticle}
\bauthor{\bsnm{Greaves}, \binits{G.R.W.}},
\bauthor{\bsnm{Iverson}, \binits{J.W.}},
\bauthor{\bsnm{Jasper}, \binits{J.}},
\bauthor{\bsnm{Mixon}, \binits{D.G.}}:
\batitle{Frames over finite fields: equiangular lines in orthogonal geometry}.
\bjtitle{Linear Algebra Appl.}
\bvolume{639},
\bfpage{50}--\blpage{80}
(\byear{2022})
\doiurl{10.1016/j.laa.2021.11.024}
\end{barticle}
\endbibitem

%%% 15
\bibitem[\protect\citeauthoryear{Elman et~al.}{2008}]{MR2427530}
\begin{bbook}
\bauthor{\bsnm{Elman}, \binits{R.}},
\bauthor{\bsnm{Karpenko}, \binits{N.}},
\bauthor{\bsnm{Merkurjev}, \binits{A.}}:
\bbtitle{The Algebraic and Geometric Theory of Quadratic Forms}.
\bsertitle{American Mathematical Society Colloquium Publications},
vol. \bseriesno{56},
p. \bfpage{435}.
\bpublisher{American Mathematical Society},
\blocation{Providence, RI}
(\byear{2008}).
\doiurl{10.1090/coll/056} .
\burl{https://doi.org/10.1090/coll/056}
\end{bbook}
\endbibitem

%%% 16
\bibitem[\protect\citeauthoryear{Grove}{2002}]{MR1859189}
\begin{bbook}
\bauthor{\bsnm{Grove}, \binits{L.C.}}:
\bbtitle{Classical Groups and Geometric Algebra}.
\bsertitle{Graduate Studies in Mathematics},
vol. \bseriesno{39},
p. \bfpage{169}.
\bpublisher{American Mathematical Society, Providence, RI}, \blocation{???}
(\byear{2002}).
\doiurl{10.1090/gsm/039} .
\burl{https://doi.org/10.1090/gsm/039}
\end{bbook}
\endbibitem

\end{thebibliography}
%% if required, the content of .bbl file can be included here once bbl is generated
%%\input sn-article.bbl

\end{document}